\documentclass[oneside,english]{amsart}
\usepackage{lmodern}
\usepackage[T1]{fontenc}
\usepackage[latin9]{inputenc}
\usepackage{amsthm}
\usepackage{amssymb}
\PassOptionsToPackage{normalem}{ulem}
\usepackage{ulem}

\makeatletter
\numberwithin{equation}{section}
\numberwithin{figure}{section}
\theoremstyle{plain}
\newtheorem{thm}{\protect\theoremname}[section]
  \theoremstyle{plain}
  \newtheorem{fact}[thm]{\protect\factname}
\newcounter{mainthm}
\theoremstyle{plain}

\newtheorem{main_thm}[mainthm]{Main Theorem}
  \theoremstyle{definition}
  \newtheorem{defn}[thm]{\protect\definitionname}
  \theoremstyle{remark}
  \newtheorem{rem}[thm]{\protect\remarkname}
  \theoremstyle{definition}
  \newtheorem{example}[thm]{\protect\examplename}
  \theoremstyle{plain}
  \newtheorem{prop}[thm]{\protect\propositionname}
  \theoremstyle{plain}
  \newtheorem{cor}[thm]{\protect\corollaryname}
  \theoremstyle{plain}
  \newtheorem{lem}[thm]{\protect\lemmaname}
  \theoremstyle{remark}
  \newtheorem{claim}[thm]{\protect\claimname}
\newcounter{assumptionctr}
\theoremstyle{plain}

\newtheorem{assumption}[assumptionctr]{Assumption}

\usepackage{url}

\usepackage{multicol}
\usepackage{a4wide}
\usepackage{amssymb}
\usepackage{url}
\makeatother

\usepackage{babel}
  \providecommand{\claimname}{Claim}
  \providecommand{\corollaryname}{Corollary}
  \providecommand{\definitionname}{Definition}
  \providecommand{\examplename}{Example}
  \providecommand{\factname}{Fact}
  \providecommand{\lemmaname}{Lemma}
  \providecommand{\propositionname}{Proposition}
  \providecommand{\remarkname}{Remark}
\providecommand{\theoremname}{Theorem}

\begin{document}
\global\long\def\Aut{\operatorname{Aut}}
\global\long\def\C{\mathfrak{C}}
\global\long\def\acl{\operatorname{acl}}
\global\long\def\tp{\operatorname{tp}}
\global\long\def\qf{\operatorname{qf}}
\global\long\def\Nn{\mathbb{N}}
\global\long\def\id{\operatorname{id}}
\global\long\def\xx{\mathbf{x}}
\global\long\def\ded{\operatorname{ded}}
\global\long\def\cof{\operatorname{cof}}
\global\long\def\SS{\mathcal{P}}
\global\long\def\EM{\operatorname{EM}}
\global\long\def\dcl{\operatorname{dcl}}
\global\long\def\Autf{\operatorname{Aut}f_{L}}
\global\long\def\eq{\operatorname{eq}}

\global\long\def\pamod#1{\pmod#1}

\global\long\def\nf{\mbox{nf}}
\global\long\def\Uu{\mathcal{U}}
\global\long\def\dom{\operatorname{dom}}
\global\long\def\concat{\smallfrown}
\global\long\def\Nn{\mathbb{N}}
\global\long\def\mathrela#1{\mathrel{#1}}
\global\long\def\twiddle{\mathord{\sim}}
 \global\long\def\x{\times}
\global\long\def\diam{\operatorname{diam}}
\global\long\def\EZero{\mathbb{E}_{0}}
\global\long\def\sequence#1#2{\left\langle #1\left|\,#2\right.\right\rangle }
\global\long\def\set#1#2{\left\{  #1\left|\,#2\right.\right\}  }
\global\long\def\cardinal#1{\left|#1\right|}
\global\long\def\calO{\mathcal{O}}

\global\long\def\NTPT{\operatorname{NTP}_{\operatorname{2}}}
\global\long\def\ist{\operatorname{ist}}
\global\long\def\C{\mathfrak{C}}
\global\long\def\alt{\operatorname{alt}}
\global\long\def\assum{\smiley}
\global\long\def\leftexp#1#2{{\vphantom{#2}}^{#1}{#2}}
\global\long\def\mathordi#1{\mathord{#1}}
\global\long\def\ord{\mathbf{ord}}

\title{The Borel cardinality of Lascar strong types}

\author{Itay Kaplan, Benjamin Miller, and Pierre Simon}

\thanks{The first and second authors' research was supported in part by SFB
grant 878. }
\begin{abstract}
We show that if the restriction of the Lascar equivalence relation
to a KP-strong type is non-trivial, then it is non-smooth (when viewed
as a Borel equivalence relation on an appropriate space of types).
\end{abstract}
\maketitle

\section{Introduction}

Notions of strong type play an important role in the study of first-order
theories. A \emph{strong type} (over $\emptyset$) is a class of an
automorphism-invariant equivalence relation on $\C^{\alpha}$ which
is bounded (i.e., the quotient has small cardinality) and refines
equality of types. The phrase ``strong type'' by itself often refers
to a \emph{Shelah strong type,} which is simply a type over the algebraic
closure of $\emptyset$ (in $T^{\eq}$). In other words, two sequences
have the same Shelah strong type if they are equivalent with respect
to every definable equivalence relation with finitely many classes.
Refining this is the notion of \emph{KP strong type} ($\equiv_{KP}^{\alpha}$),
in which two sequences are equivalent if they are equivalent with
respect to every bounded type-definable equivalence relation. The
KP strong type can also be characterized as the finest notion of strong
type for which the corresponding quotient is a compact Hausdorff space.
Finally, the \emph{Lascar strong type} ($\equiv_{L}^{\alpha}$) is
the finest notion of strong type. The classes of $\equiv_{L}^{\alpha}$
coincide with the connected components of the \emph{Lascar graph}
on $\C^{\alpha}$, in which two sequences are neighbors if they lie
along an infinite indiscernible sequence. The \emph{Lascar distance}
$d$ is the associated graph distance. All of this is explained in
detail in Subsection \ref{sub:Model-theoretic-preliminaries}.

In \cite{Newelski}, Newelski established the following fundamental
facts: 
\begin{fact}
\label{fac:Newelski}\cite{Newelski} Suppose that $T$ is a complete
first-order theory and $\alpha$ is an ordinal. 
\begin{enumerate}
\item A Lascar strong type is type definable iff it has finite diameter.
\item If $Y$ is an $\mathordi{\equiv_{L}^{\alpha}}$-invariant closed set,
contained in some $p\in S_{\alpha}\left(\emptyset\right)$, on which
every $\mathordi{\equiv_{L}^{\alpha}}$-class has infinite diameter,
then $Y$ contains at least $2^{\aleph_{0}}$-many $\mathordi{\equiv_{L}^{\alpha}}$-classes. 
\item Lascar strong types of unbounded diameter are not $G_{\delta}$ sets
(when viewed as subsets of an appropriate space of types, as explained
in Subsection \ref{sub:Context}).
\item If $T$ is small (i.e., $T$ is countable and the number of finitary
types over $\emptyset$ is countable), then $\mathordi{\equiv_{L}^{n}}=\mathordi{\equiv_{KP}^{n}}$
for all $n<\omega$ (i.e., the two notions of type agree on finite
sequences).
\end{enumerate}
\end{fact}
As opposed to Shelah and KP strong types, the space of Lascar strong
types does not come equipped with a Hausdorff topology. It is therefore
unclear to what category this quotient belongs. In \cite{PillayKrupinskiSolecki},
the authors suggest viewing it through the framework of descriptive
set theory (this idea was already mentioned in \cite{CasanovasPillayLascarZiegler}).
They formally interpret the notion of equality of Lascar-strong types
as a Borel equivalence relation over a compact Polish space, and then
consider the position of this relation in the Borel reducibility hierarchy.

Given two Polish spaces $X$ and $X'$ and two Borel equivalence relations
$E$ and $E'$ respectively on $X$ and $X'$, we say that $E$ is
\emph{Borel reducible} to $E'$ if there is a Borel map $f$ from
$X$ to $X'$ such that $x\mathrela Ey\iff f(x)\mathrela{E'}f(y)$
for all $x,y\in X$. Two relations are \emph{Borel bi-reducible} if
each is Borel reducible to the other. The quasi-order of Borel reducibility
is a well-studied object in descriptive set theory, and enjoys a number
of remarkable properties. One of them is given by the Harrington-Kechris-Louveau
dichotomy, which asserts that a Borel equivalence relation is either
smooth (Borel reducible to equality on $2^{\omega}$) or at least
as complicated as $\EZero$ (eventual equality on $2^{\omega}$).
This is explained in detail in Subsections \ref{sub:Preliminaries-on-definable}
and \ref{sub:Preliminaries-on-definable}. 

In this paper, we provide the following solution to the main conjecture
of \cite{PillayKrupinskiSolecki}. 
\begin{main_thm}
{[}Simplified version{]}\label{MainThmA-1} Suppose that $T$ is a
complete countable first-order theory. If $\equiv_{L}$ does not coincide
with $\equiv_{KP}$, then $\equiv_{L}$ is not smooth.
\end{main_thm}
\setcounter{mainthm}{0} 

We will actually prove a slightly stronger result (see Theorem \ref{thm:T is not smooth}).
Our proof will not make use of Fact \ref{fac:Newelski}, which we
will recover (for countable $T$) as a corollary.

Let us say a few words about our method. In \cite{NewelskiPetrykowski},
Newelski and Petrykowski introduce the notion of weakly generic types
for definable groups. An analog for groups of automorphisms was used
in Pelaez's thesis \cite{rodrigothesis} to give an alternate proof
of Fact \ref{fac:Newelski} (1). We follow this lead in our own proof.

A consequence of the proof of an early special case of the Harrington-Kechris-Louveau
dichotomy theorem is that if $X$ is a Polish space, $G$ is a Polish
group acting continuously on $X$, and the orbit equivalence relation
$E_{G}$ is $F_{\sigma}$, then either $E_{G}$ is smooth or $\EZero$
can be continuously embedded into $E_{G}$. This is related to \cite[Theorem 3.4.5]{BeckerKechris},
which gives a sufficient condition for embedding $\EZero$ in an equivalence
relation induced by a group action (namely that there is a dense orbit
and that $E_{G}$ is meager). 

As a corollary of the latter result, we establish a sufficient criterion
for embedding $\EZero$ into an equivalence relation $E$ whose classes
are each equipped with a metric. Roughly speaking, the group $G$
of homeomorphisms of the Cantor space whose graphs are contained in
$\EZero$ acts in a sufficiently rich fashion that it can move any
element to another which is arbitrarily close in the topological sense,
but far away in the sense of the metric associated with the class.
We show that if a similar property holds for $E$ then one can embed
$\EZero$ into $E$. 

Assuming that there is a Lascar-strong type of unbounded diameter,
it is thus natural to try to find a type $p$ whose orbit under the
group of Lascar-strong automorphisms is also sufficiently rich. When
$T$ is countable, we construct such a type formula-by-formula. At
each stage, we must make sure that we still have room to go on, namely,
that the partial type has many images which are at large Lascar distance
from each other. To this end, we make sure that the type stays weakly
generic. We actually need a slightly stronger property which we call
``properness''.

Main Theorem \ref{MainThmA} does not seem to be enough to deduce
Fact \ref{fac:Newelski} for uncountable theories. However, we adapt
our argument to also take care of this case. This is Main Theorem
\ref{MainThmB}. For uncountable languages, the space of types is
no longer a Polish space, so we do not state the theorem in terms
of Borel cardinality. Apart from that, the result is essentially the
same as for countable theories. In particular, it implies Newelski's
theorem. In order to prove it, we will need a little bit more from
the descriptive set theoretic side, namely, the notion of (strong)
Choquet space. This is used to replace completeness. In fact, eventually
we deal with a non-Hausdorff space. (A non-Hausdorff space will arise
as the space of types over a model $M$ with the topology induced
by a countable sub-language $L'$ of $L$ and a corresponding countable
model $M'$.)

\subsubsection*{Organization of the paper}

We have made an effort to keep this paper self-contained and accessible
to model theorists and descriptive set theorists alike. Thus we start
by giving all of the required definitions from both sides. In Section
\ref{sec:Descriptive-set-theoretic}, we state a set theoretic criterion
for non-smoothness. In Section \ref{sec:The-small-case}, we treat
a special case of the main theorem, where $T$ is small (hence Lascar
strong types coincide with KP strong types on finite tuples) and $\alpha$
is infinite. Although this result will not be used to prove the general
case, we thought it worthwhile to include, as the proof is considerably
simpler and gives insight into the general case. In Section \ref{sec:The-general-case},
we prove Main Theorem \ref{MainThmA-1} for all countable theories.
Finally, in section \ref{sec:Uncountable-language}, we prove Main
Theorem \ref{MainThmB}, thereby taking care of the general case.

\subsection{\label{sub:Model-theoretic-preliminaries}\textmd{ }Model-theoretic
preliminaries. }

Let $T$ be any complete first order theory. The theory $T$ may be
many sorted, but for the simplicity of the presentation one may assume
that it is one sorted. We recall some basic notions. We fix a sequence
of variables $\sequence{v_{i}}{i\in\ord}$. For the rest of this section,
$\alpha$ will be some ordinal. 
\begin{defn}
Suppose $M\models T$, $A\subseteq M$. Let $L_{\alpha}\left(A\right)$
the set of formulas in variables $\sequence{v_{i}}{i<\alpha}$ with
parameters from $A$ --- the set of \emph{formulas over $A$}. An
\emph{$\alpha$-type over $A$} (sometimes a \emph{partial} $\alpha$-type
over $A$) is a consistent subset of $L_{\alpha}\left(A\right)$.
An $\alpha$-type $p$ over $A$ is called \emph{complete} if for
any formula $\varphi\in L_{\alpha}\left(A\right)$, $\left\{ \varphi,\neg\varphi\right\} $
intersects $p$. We denote by $S_{\alpha}\left(A\right)$ the space
of all complete $\alpha$-types over $A$. For a tuple $a\in M^{\alpha}$,
let 
\[
\tp\left(a/A\right)=\set{\varphi\in L_{\alpha}\left(A\right)}{M\models\varphi\left(a\right)}.
\]
We write $a\equiv_{A}b$ for $\tp\left(a/A\right)=\tp\left(b/A\right)$.
If $A=\emptyset$, we omit it. If $p$ is an $\alpha$-type over $A$
and $\tp\left(a/A\right)\supseteq p$ we write $a\models p$. We say
that $p$ is \emph{realized} in $M$ if there exists some $a\in M^{\alpha}$
such that $a\models p$. 

We sometimes write $p\left(x\right)$ (respectively $\varphi\left(x\right)$)
when we want to stress that the free variables of $p$ (respectively
$\varphi$) are contained in the tuple of variables $x$. \end{defn}
\begin{rem}
\label{rem:stone-topology}The set $S_{\alpha}\left(A\right)$ is
naturally a compact Hausdorff topological space,\emph{ }the\emph{
Stone space} of $\alpha$-types over $A$, with clopen basis $\set{\left[\varphi\right]}{\varphi\in L_{\alpha}\left(A\right)}$,
where $\left[\varphi\right]=\set{p\in S_{\alpha}\left(A\right)}{\varphi\in p}$.\end{rem}
\begin{defn}
For a cardinal $\kappa$, a model $M\models T$ is called \emph{$\kappa$-saturated}
when for all sets $A\subseteq M$ of cardinality $<\kappa$ all types
in $S_{1}\left(A\right)$ are realized in $M$. The model $M$ is
\emph{$\kappa$-homogeneous} when for all $\alpha<\kappa$ and $a,b\in M^{\alpha}$
if $a\equiv b$ then there is an automorphism of $M$ mapping $a$
to $b$.
\end{defn}
Recall that the cardinality of $T$, $\left|T\right|$, is identified
with the cardinality of the set of formulas in $T$. 
\begin{fact}
\cite[Theorem 10.2.1]{Hod} For any cardinal $\kappa\geq\left|T\right|$,
and any model $M\models T$ there exists a $\kappa$-saturated, $\kappa$-homogeneous
model $N\succ M$ of $T$. 
\end{fact}
Fix some $\kappa\geq\left(2^{\left|T\right|}\right)^{+}$. We denote
by $\C$ a $\kappa$-saturated, $\kappa$-homogeneous model of $T$.
In model theory, this is usually referred to as the ``monster model''
of $T$ (and it is often harmless to assume in addition that $\C$
is $\left|\C\right|$-saturated and is very big). The convention is
that all sets of parameters and tuples we deal with are small, that
is, of cardinality $<\kappa$, and that they are all contained in
$\C$. Similarly all small models are assumed to be elementary substructures
of $\C$. 

Recall that for $A\subseteq\C$, $\Aut\left(\C/A\right)$ denotes
the group of automorphisms of $\C$ which fix $A$ pointwise. 
\begin{defn}
Let $A\subseteq\C$ a small set. A set $X\subseteq\C^{\alpha}$ is
called\emph{ $A$-type-definable} (or \emph{type-definable over $A$})
if it is empty or there is an $\alpha$-type $p$ over $A$ such that
\[
X=\set{a\in\C^{\alpha}}{a\models p}.
\]
 It is \emph{$A$-invariant} (or \emph{invariant over $A$}) when
for all $\sigma\in\Aut\left(\C/A\right)$, $\sigma^{\alpha}\left(X\right)=X$
(usually we omit $\alpha$ from this notation). When $A$ is omitted,
it is understood that $A=\emptyset$. 
\end{defn}
We define a ``topology'' on subsets of $\C^{\alpha}$.
\begin{defn}
\label{def:psuedo topology}Call a subset $X\subseteq\C^{\alpha}$
\emph{pseudo closed} if $X$ is type definable over some small set.
A  \emph{pseudo open} set is a complement of a pseudo closed set.
\emph{Pseudo $G_{\delta}$} sets and \emph{pseudo $F_{\sigma}$} sets
are defined in the obvious way. 
\end{defn}
By saturation $\C^{\alpha}$ is pseudo compact in the sense that any
small intersection of non-empty pseudo closed sets is non-empty. This
why we often say ``by compactness'', instead of ``by saturation''. 
\begin{rem}
\label{rem:restriction map}By compactness, for a small set $A\subseteq\C$,
the map $r_{\alpha,A}:\C^{\alpha}\to S_{\alpha}\left(A\right)$ defined
by $a\mapsto\tp\left(a/A\right)$ is pseudo closed, in the sense that
it sends pseudo closed sets to closed sets (in the Stone topology).
So $r_{\alpha,A}$ maps pseudo $F_{\sigma}$ subsets of $\C^{\alpha}$
to $F_{\sigma}$ subsets of $S_{\alpha}\left(A\right)$. 
\end{rem}
We also recall the notion of an indiscernible sequence:
\begin{defn}
Let $A$ be a small set. Let $\left(I,<\right)$ be some linearly
ordered set. A sequence $\bar{a}=\sequence{a_{i}}{i\in I}\in\left(\C^{\alpha}\right)^{I}$
is called \emph{$A$-indiscernible} (or \emph{indiscernible over $A$})
if for all $n<\omega$, every increasing $n$-tuple from $\bar{a}$
realizes the same type over $A$. When $A$ is omitted, it is understood
that $A=\emptyset$. 
\end{defn}
An easy but very important fact about indiscernible sequences is that
they exist. 
\begin{fact}
\label{fac:indiscernibles exist}$ $
\begin{enumerate}
\item \cite[Lemma 5.1.3]{TentZiegler} Let $\left(I,<_{I}\right)$, $\left(J,<_{J}\right)$
be small linearly ordered sets, and let $A$ be some small set. Suppose
$\bar{b}=\sequence{b_{j}}{j\in J}$ is some sequence of elements from
$\C^{\alpha}$. Then there exists an indiscernible sequence $\bar{a}=\sequence{a_{i}}{i\in I}\in\left(\C^{\alpha}\right)^{I}$
such that:

\begin{itemize}
\item For any $n<\omega$ and $\varphi\in L_{\alpha\cdot n}$, if $\C\models\varphi\left(b_{j_{0}},\ldots,b_{j_{n-1}}\right)$
for every $j_{0}<_{J}\ldots<_{J}j_{n-1}$ from $J$ then $\C\models\varphi\left(a_{i_{0}},\ldots,a_{i_{n-1}}\right)$
for every $i_{0}<_{I}\ldots<_{I}i_{n-1}$ from $I$.
\end{itemize}
\item \cite[proof of Proposition 3.1.4]{AvivThesis} If $M$ is a small
model and $a\equiv_{M}b$, then there is an indiscernible sequence
$\bar{c}=\sequence{c_{i}}{i<\omega}$ such that both $a\concat\bar{c}$
and $b\concat\bar{c}$ are indiscernible. 
\end{enumerate}
\end{fact}
Point (1) in Fact \ref{fac:indiscernibles exist} is proved using
Ramsey's theorem and compactness, while (2) is proved with ultrafilters. 
\begin{defn}
An equivalence relation $E$ on a set $X$ is called \emph{bounded}
if $\left|X/E\right|<\kappa$. \end{defn}
\begin{rem}
\label{rem:disscussion of bounded eq. relations}By saturation and
homogeneity, every invariant set is a union of types over $\emptyset$.
So by saturation if $E$ is an invariant equivalence relation with
an invariant domain $X\subseteq\C^{\alpha}$, it makes sense to consider
$E$ in any monster model. When $E$ is bounded, and $\left|\alpha\right|\leq\left|T\right|$
then $\left|X/E\right|\leq2^{\left|T\right|}$. To see that, let $M\models T$
be of size $\left|T\right|$. If $a,b\in X$ and $a\equiv_{M}b$ then
$\left(a,b\right)\in E$, since otherwise, by Fact \ref{fac:indiscernibles exist}
(2) and saturation, we may assume that $\left\langle a,b\right\rangle $
starts an indiscernible sequence of length $\kappa$. By homogeneity,
any two elements in it are not $E$-equivalent. Now the result follows
from the fact that $\left|S_{\alpha}\left(M\right)\right|\leq2^{\left|T\right|}$.
It is now easy to see that if $\C'\succ\C$ is another monster model
then every $E$-class in $\C'$ intersects $\C$, so there are no
``new'' classes. 
\end{rem}
We come to the central definition.
\begin{defn}
\label{def:Lascar distance}The\emph{ Lascar graph} on $\C^{\alpha}$
is the set $G_{\alpha}$ of pairs $\left(a,b\right)$ of distinct
elements of $\C^{\alpha}$ which lie along an infinite indiscernible
sequence. The Lascar metric $d_{\alpha}$ is the metric associated
with this graph. Let $\mathordi{\equiv_{L}^{\alpha}}$ denote the
equivalence relation on $\C^{\alpha}$ whose classes coincide with
the connected components of $G_{\alpha}$. The \emph{Lascar strong
type }of a tuple $a\in\C^{\alpha}$ is its $\mathordi{\equiv_{L}^{\alpha}}$-class.
We will omit $\alpha$ from the notation when it is clear from context. \end{defn}
\begin{rem}
\label{rem:distance 2}By Fact \ref{fac:indiscernibles exist} (2),
it follows that for $a,b\in\C^{\alpha}$ and $M\prec\C$, if $a\equiv_{M}b$
then $d\left(a,b\right)\leq2$. \end{rem}
\begin{fact}
(see e.g., \cite[Proposition 3.1.4]{AvivThesis}) The relation $\mathordi{\equiv_{L}^{\alpha}}$
is the finest bounded invariant equivalence relation on $\C^{\alpha}$.\end{fact}
\begin{proof}
[Proof (sketch.)] If $E$ is some bounded invariant equivalence relation
on $\C^{\alpha}$ and $d_{\alpha}\left(a,b\right)\leq1$, then as
in Remark \ref{rem:disscussion of bounded eq. relations}, $\left(a,b\right)\in E$.
Similarly, $\mathordi{\equiv_{L}^{\alpha}}$ is bounded since it is
bounded by $\left|S_{\alpha}\left(M\right)\right|$ for any model
$M\models T$. \end{proof}
\begin{defn}
\label{def:Autfl}The group of \emph{Lascar strong automorphisms}
of $\C$ is the group generated by automorphisms $\sigma$ of $\C$
for which there is a small model $M\prec\C$ fixed pointwise by $\sigma$,
i.e., the group 
\[
\Autf\left(\C\right)=\left\langle \sigma\in\Aut\left(\C/M\right)\left|\, M\prec\C\right.\right\rangle .
\]
\end{defn}
\begin{fact}
\label{fac:Lascar Strong types, automorphisms, connection}(see e.g.,
\cite[Section 3.1]{AvivThesis}) 
\begin{enumerate}
\item The group of Lascar strong automorphisms is a normal subgroup of $\Aut\left(\C\right)$.
It consists of all automorphisms that fix all Lascar strong types
(of any length). 
\item The Lascar strong type equivalence relation is the orbit equivalence
relation of the group of Lascar strong automorphisms.
\item If $\sigma$ is a Lascar strong automorphism, then there is some $m<\omega$
such that for any tuple $c$ (of any length), $d\left(c,\sigma\left(c\right)\right)\leq m$.
In this case we say that $m$ \uline{bounds} $\sigma$.
\end{enumerate}
\end{fact}

\begin{rem}
\label{rem:distance at most doubled} Suppose $a,b\in\C^{\alpha}$
and $d_{\alpha}\left(a,b\right)\leq n$. Then there is a Lascar strong
automorphism $\sigma$ of $\C$ bounded by $2n$ such that $\sigma\left(a\right)=b$. \end{rem}
\begin{proof}
(of Remark \ref{rem:distance at most doubled}) It is enough to establish
it in the case $d\left(a,b\right)\leq1$: if $d\left(a,b\right)\leq n$,
then there are $c_{0},\ldots,c_{n}$ with $a=c_{0}$, $c_{n}=b$ and
$d\left(c_{i},c_{i+1}\right)\leq1$ for all $i<n$. For each $i<n$,
we find some $\sigma_{i}$ bounded by $2$ that maps $c_{i}$ to $c_{i+1}$.
Let $\sigma=\sigma_{n-1}\circ\ldots\circ\sigma_{0}$. 

So suppose $I=\left\langle a_{i}\left|\, i<\omega\right.\right\rangle $
is an indiscernible sequence that starts with $a_{0}=a,a_{1}=b$.
Let $M$ be a model of size $\left|T\right|$. By saturation we can
extend the sequence $I$ to length $\left(2^{\left|T\right|}\right)^{+}$.
So there must be two elements in $I$ that have the same type over
$M$. By indiscernibility and homogeneity, there is some model $M'$
such that $a\equiv_{M'}b$.

The remark now follows from Remark \ref{rem:distance 2}. 
\end{proof}
We also recall the notion of KP strong type:
\begin{defn}
Let $\mathordi{\equiv_{KP}^{\alpha}}$ denote the finest bounded type-definable
equivalence relation on $\C^{\alpha}$. The \emph{KP strong type}%
\footnote{\emph{KP stands for Kim-Pillay. This notation was introduced by Hrushovski
in \cite{Hrushovski98simplicityand}.}%
}\emph{ }of a tuple $a\in\C^{\alpha}$ is its $\mathordi{\equiv_{KP}^{\alpha}}$-class. \end{defn}
\begin{fact}
\label{fac:KP is still finest} \cite[Proposition 15.25]{cassanovas-simple}
Let $X$ be any type-definable subset of $\C^{\alpha}$. 
\begin{enumerate}
\item The restriction $\mathordi{\equiv_{_{L}}^{\alpha}}\upharpoonright X$
of $\mathordi{\equiv_{_{L}}^{\alpha}}$ to $X$ is the finest bounded
invariant equivalence relation on realizations of $X$. 
\item The restriction $\mathordi{\equiv_{_{KP}}^{\alpha}}\upharpoonright X$
of $\mathordi{\equiv_{_{KP}}^{\alpha}}$ is the finest bounded type-definable
equivalence relation on realizations of $p$.
\end{enumerate}
\end{fact}
\begin{rem}
\label{rem:Lascar is not closed} By saturation and homogeneity if
$X\subseteq\C^{\alpha}$ is type-definable over some small set $B$
and invariant over another small set $A$, then it is type-definable
over $A$. It follows that if $K\subseteq\C^{\alpha}$ is a KP strong
type, and for some $a\in K$, $\left[a\right]_{\mathordi{\equiv_{L}^{\alpha}}}$
is pseudo closed, then $\mathordi{\equiv_{L}^{\alpha}}\upharpoonright K$
is trivial. Indeed, it is type-definable over $a$ so there is a type
$\pi\left(x,y\right)$ such that $\pi\left(x,a\right)$ defines $\left[a\right]_{\mathordi{\equiv_{L}}}$.
Let $p\left(x\right)=\tp\left(a/\emptyset\right)$. Then $\mathordi{\equiv_{L}^{\alpha}}\upharpoonright p$
is defined by: $x\equiv_{L}^{\alpha}y$ iff $\pi\left(x,y\right)$.
Fact \ref{fac:KP is still finest} implies that $\mathordi{\equiv_{L}^{\alpha}}\upharpoonright p=\mathordi{\equiv_{KP}^{\alpha}}\upharpoonright p$,
so $\mathordi{\equiv_{L}^{\alpha}}\upharpoonright K$ is trivial. \end{rem}
\begin{defn}
Let $Y\subseteq\C^{\alpha}$ be closed under $\mathordi{\equiv_{L}^{\alpha}}$.
We say that $Y$ is \emph{$d$-bounded} if there is some $n<\omega$
such that $a\equiv_{L}^{\alpha}b$ iff $d\left(a,b\right)\leq n$
for all $a,b\in Y$. \end{defn}
\begin{rem}
For a set of parameters $A$, \emph{Lascar distance over $A$}, \emph{Lascar
strong type over $A$}, \emph{KP-strong type over $A$}, etc., are
the parallel notions for $T_{A}$: the complete theory of the structure
$\C_{A}$ which is just $\C$ after naming all elements from $A$.
All the facts above hold for $A$ with the obvious adjustments. 
\end{rem}

\subsection{\label{sub:Preliminaries-on-definable}Preliminaries on Borel equivalence
relations.}

Here we give the basic facts about Borel equivalence relations. 
\begin{defn}
Suppose $X$ and $Y$ are Polish spaces, and $E$ and $F$ are Borel
equivalence relations on $X$ and $Y$. We say that a function $f:X\to Y$
is a \emph{reduction} of $E$ to $F$ if for all $x_{0},x_{1}\in X$,
$\left(x_{0},x_{1}\right)\in E$ iff $\left(f\left(x_{0}\right),f\left(x_{1}\right)\right)\in F$.
\begin{enumerate}
\item We say that $E$ is\emph{ Borel reducible} to $F$, denoted by $E\leq_{B}F$,
when there is a Borel reduction $f:X\to Y$ of $E$ to $F$.
\item We say that $E$ is \emph{continuously reducible} to $F$, denoted
by $E\sqsubseteq_{c}F$, when there is a continuous injective reduction
$f:X\to Y$ of $E$ to $F$.
\item We say that $E$ and $F$ are \emph{Borel bi-reducible}, denoted by
$E\sim_{B}F$, when $E\leq_{B}F$ and $F\leq_{B}E$. 
\item We write $E<_{B}F$ to mean that $E\leq_{B}F$ but $E\not\sim_{B}F$. 
\end{enumerate}
\end{defn}
\begin{example}
For a Polish space $X$, the relations $\Delta\left(X\right)$ denotes
equality on $X$. Then $\Delta\left(1\right)<_{B}\Delta\left(2\right)<_{B}\ldots<_{B}\Delta\left(\omega\right)<_{B}\Delta\left(2^{\omega}\right)$. \end{example}
\begin{defn}
We say that $E$ is \emph{smooth} iff $E\leq_{B}\Delta\left(2^{\omega}\right)$. 
\end{defn}
Note that being smooth is equivalent to the existence of ``separating
Borel sets,'' i.e., Borel sets $B_{i}\subseteq X$ such that $x\mathrela Ey$
iff for all $i<\omega$, $x\in B_{i}$ iff $y\in B_{i}$. 
\begin{fact}
\cite{silver}(Silver dichotomy) For all Borel equivalence relations
$E$, $E\leq_{B}\Delta\left(\omega\right)$ or $\Delta\left(2^{\omega}\right)\sqsubseteq_{c}E$
. It follows that $\Delta\left(2^{\omega}\right)$ is the successor
of $\Delta\left(\omega\right)$.\end{fact}
\begin{prop}
\label{prop:Closed-equivalence-relations}Closed equivalence relations
are smooth. \end{prop}
\begin{proof}
Suppose $E$ is a closed equivalence relation on a Polish space $X$.
We must find Borel set $B_{i}\subseteq X$ for $i<\omega$ such that
$xEy$ iff for all $i<\omega$, $x\in B_{i}\Leftrightarrow y\in B_{i}$.
Since $X^{2}\backslash E$ is open, it equals $\bigcup_{i<\omega}U_{i}\x V_{i}$
for $U_{i},V_{i}\subseteq X$ open. Let $U_{i}^{E}=\left\{ x\in X\left|\,\exists y\left(y\in U_{i}\,\&\, xEy\right)\right.\right\} $
be the $E$-closure of $U_{i}$ and $V_{i}^{E}$ be the $E$-closure
of $V_{i}$. These are analytic sets. Since $U_{i}^{E}\cap V_{i}^{E}=\emptyset$,
by Lusin's separation theorem, there are Borel sets $U_{i}^{0}$ such
that $U_{i}^{0}\supseteq U_{i}^{E}$, $U_{i}^{0}\cap V_{i}^{E}=\emptyset$.
Recursively we construct Borel sets $U_{i}^{j}$ for $j<\omega$ such
that $U_{i}^{j}$ contains the $E$-closure of $U_{i}^{j-1}$ and
is disjoint from $V_{i}^{E}$. Let $B_{i}=\bigcup_{j<\omega}U_{i}^{j}$. \end{proof}
\begin{example}
Let $\EZero$ be the following equivalence relation on the Cantor
space $2^{\omega}$: $\left(\eta,\nu\right)\in\EZero$ iff there exists
some $n<\omega$ such that for all $m>n$, $\eta\left(m\right)=\nu\left(m\right)$. \end{example}
\begin{prop}
\label{prop:E_0 is not definable smooth}The relation $\EZero$ is
non-smooth.\end{prop}
\begin{proof}
Recall that all Borel subsets $B$ of a Polish space $X$ have the
Baire property: there is an open set $O\subseteq X$ such that $O\Delta B$
is meager. Suppose $\set{B_{i}}{i\in\omega}$ are Borel separating
sets of $\EZero$, so all of them have the Baire property. 

Fix some $i<\omega$, and suppose $B_{i}$ is not meager. Then there
is some $n<\omega$ and some $s\in2^{n}$ such that, letting $O_{s}=\set{\eta\in2^{\omega}}{s\triangleleft\eta}$,
$O_{s}\backslash B_{i}$ is meager. Let $t\in2^{n}$. Since $B_{i}$
is closed under $\EZero$, and there is a homeomorphism of $2^{\omega}$
taking $O_{s}$ to $O_{t}$ fixing all $\EZero$-classes, $O_{t}\backslash B_{i}$
is also meager. But then $2^{\omega}\backslash B_{i}=\bigcup_{s\in2^{n}}O_{s}\backslash B_{i}$
is meager, so $B_{i}$ is comeager. This shows that $B_{i}$ is either
meager or comeager.

But then, 
\[
B=\bigcap\set{B_{i}}{i<\omega,\, B_{i}\mbox{ is comeager}}\cap\bigcap\set{\twiddle B_{i}}{i<\omega,\, B_{i}\mbox{ is meager}}
\]
 is a comeager $\EZero$-class, which is a contradiction (since it
is countable). 
\end{proof}
In addition, we have the following dichotomy:
\begin{fact}
\label{fac:Harrington-Kechris-Louveau-dich} \cite{harringtonKechrisLouveau}
(Harrington-Kechris-Louveau dichotomy) For every Borel equivalence
relation $E$ either $E\leq_{B}\Delta\left(2^{\omega}\right)$ (i.e.,
$E$ is smooth) or $\EZero\sqsubseteq_{c}E$. It follows that $\EZero$
is the successor of $\Delta\left(2^{\omega}\right)$. 
\end{fact}
We also mention:
\begin{cor}
\label{cor:G_delta=00003D>smooth}Suppose $Y$ is a Polish space,
and $E$ is a Borel equivalence relation on $Y$ such that all its
classes are $G_{\delta}$-subsets. Then $E$ is smooth.\end{cor}
\begin{proof}
Suppose $E$ is not smooth. By Fact \ref{fac:Harrington-Kechris-Louveau-dich},
there is a continuous map $f:2^{\omega}\to Y$ that reduces $\EZero$
to $E$. But then it follows that all the $\EZero$-classes are continuous
pre-images of $G_{\delta}$ sets, so they are themselves $G_{\delta}$.
As they are also dense, this is a contradiction. 
\end{proof}

\subsection{\label{sub:Preliminaries-on-Choquet}Preliminaries on Choquet spaces.}

As we mentioned above, when the language is not necessarily countable
we will work with Choquet spaces instead of Polish spaces.
\begin{defn}
The \emph{Choquet game} on a topological space $X$ is a two player
game in $\omega$-many rounds. In round $n$, player A chooses a non-empty
open set $U_{n}\subseteq V_{n-1}$ (where $V_{-1}=X$), and player
B responds by choosing a non-empty open subset $V_{n}\subseteq U_{n}$.
Player B wins if the intersection $\bigcap\set{V_{n}}{n<\omega}$
is not empty. 

The \emph{strong Choquet game }is similar: in round $n$ player A
chooses an open set $U_{n}\subseteq V_{n-1}$ and $x_{n}\in U_{n}$,
and player B responds by choosing an open set $V_{n}\subseteq U_{n}$
containing $x_{n}$. Again, player B wins when the intersection $\bigcap\set{V_{n}}{n<\omega}$
is not empty. 

A topological space $X$ is a \emph{(strong) Choquet space} if player
B has a winning strategy in every (strong) Choquet game. 

Given a subset $A$ of $X$, we say that $X$ is \emph{strong Choquet
over $A$} to mean that the points that player A chooses are taken
from $A$.
\end{defn}
It is easy to see that:
\begin{example}
Every Polish space is strong Choquet. 
\end{example}
But for our purposes, we shall need the following example:
\begin{example}
If $X$ is compact (not necessarily Hausdorff) and has a basis consisting
of clopen sets then it is strong Choquet.\end{example}
\begin{proof}
In round $n$, player B will choose a clopen set $x_{n}\in V_{n}\subseteq U_{n}$.
By compactness, the intersection $\bigcap\set{V_{n}}{n<\omega}$ is
not empty. \end{proof}
\begin{prop}
\label{prop:G_delta subset is Choquet}If $X$ is strong Choquet and
$\emptyset\neq U\subseteq X$ is $G_{\delta}$, then $U$ is also
strong Choquet.\end{prop}
\begin{proof}
Suppose $U=\bigcap\set{W_{n}}{n<\omega}$ where $W_{n}\subseteq X$
are open. Let $St$ be a strategy for the strong Choquet game in $X$
and we will describe a strategy $St_{U}$ for the strong Choquet game
in $U$. So we play a game $\Game_{U}$ in $U$, and we run a parallel
game $\Game_{X}$ in $X$ as follows. Assume we have already played
all the rounds up to $n$: the sets $U_{i},V_{i}$ were chosen for
$i<n$ in the game $\Game_{U}$, and $U'_{i},V'_{i}$ are the corresponding
moves in the $\Game_{X}$. The construction will ensure that for all
$i<n$, we have $U'_{i}\cap U=U_{i}$, $V'_{i}\cap U=V_{i}$ and $U'_{n}\subseteq W_{n}$.
Assume that A plays $(U_{n},x_{n})$, with $x_{n}\in U_{n}$. Pick
an open subset $U_{*}$ of $X$ such that $U_{*}\cap U=U_{n}$. We
set A's move in the parallel game to be $(U_{*}\cap W_{n}\cap V'_{n-1},x_{n})$.
Let $V'_{n}$ be B's move according to the strategy $St$. Then in
$\Game_{U}$, have B play $V'_{n}\cap U$. Note that this set is non-empty
since it contains $x_{n}$. This defines a winning strategy for $B$.
\end{proof}

\subsection{\label{sub:Context}Context.}

\subsubsection{Countable language.}

In \cite{PillayKrupinskiSolecki}, the authors gave a natural way
of considering $\mathordi{\equiv_{L}^{\alpha}}$ and $\mathordi{\equiv_{KP}^{\alpha}}$
for a countable complete first order theory $T$ and a countable $\alpha$
as Borel equivalence relations on the space of types $S_{\alpha}\left(M\right)$
over a countable model $M$ (this is a Polish space --- see Remark
\ref{rem:stone-topology} about the topology). Fix some countable
$T$ and $\alpha$. 
\begin{defn}
Let $M$ be a countable model. For $p,q\in S_{\alpha}\left(M\right)$,
we write $p\equiv_{L}^{\alpha,M}q$ iff $\exists a\models p,b\models q\,\left(a\equiv_{L}^{\alpha}b\right)$
and similarly we define $\mathordi{\equiv_{KP}^{\alpha,M}}$. 
\end{defn}
It will be useful to define the Lascar metric on types:
\begin{defn}
\label{def:distance on types} For $p,q\in S_{\alpha}\left(M\right)$
let $d_{\alpha}\left(p,q\right)=\min\set{n\in\Nn}{\exists a\models p,b\models q\,\left(d_{\alpha}\left(a,b\right)\leq n\right)}$.
\end{defn}
Note that:
\begin{rem}
\label{rem:exists =00003D for all}\cite[Remark 2.2]{PillayKrupinskiSolecki}
Let $M$ be a countable model. By Remark \ref{rem:distance 2}, for
$p,q\in S_{\alpha}\left(M\right)$, $p\equiv_{L}^{\alpha,M}q$ iff
$\forall a\models p,b\models q\,\left(a\equiv_{L}^{\alpha}b\right)$
and similarly for $\mathordi{\equiv_{KP}^{\alpha,M}}$. 

Let $q_{\alpha,M}:S_{\alpha\cdot2}\left(M\right)\to S_{\alpha}\left(M\right)$
be defined by $p\left(x,y\right)\mapsto\left(p\upharpoonright x,q\upharpoonright y\right)$.
This is a continuous map, and hence it is closed. Using this notation,
$\mathordi{\equiv_{KP}^{\alpha,M}}=q_{\alpha,M}\circ r_{\alpha\cdot2,M}\left(\mathordi{\equiv_{KP}^{\alpha}}\right)$
(see Remark \ref{rem:restriction map}), and hence $\mathordi{\equiv_{KP}^{\alpha,M}}$
is closed. Similarly, the set 
\[
F_{n}=\left\{ \left(p,q\right)\in S_{\alpha}\left(M\right)\left|\, d_{\alpha}\left(p,q\right)\leq n\right.\right\} 
\]
 is closed, and $\mathordi{\equiv_{L}^{\alpha,M}}$ is the union $\bigcup_{n<\omega}F_{n}$
hence it is $K_{\sigma}$.
\end{rem}
They proved that as far as Borel cardinality goes, this does not depend
on the model $M$, even when restricting to a KP strong type:
\begin{fact}
\label{fac:Change the model}\cite[Propositions 2.3, 2.6]{PillayKrupinskiSolecki}
Let $M$ and $N$ be any countable models. Then,
\begin{enumerate}
\item $\mathordi{\equiv_{L}^{\alpha,M}}\sim_{B}\mathordi{\equiv_{L}^{\alpha,N}}$.
\item For any $a\in\C$, $\mathordi{\equiv_{L}^{\alpha,M}}\upharpoonright\left[\tp\left(a/M\right)\right]_{\mathordi{\equiv_{KP}^{\alpha,M}}}\sim_{B}\mathordi{\equiv_{L}^{\alpha,N}}\upharpoonright\left[\tp\left(a/N\right)\right]_{\mathordi{\equiv_{KP}^{\alpha,N}}}$. 
\end{enumerate}
\end{fact}
One can extend this observation to deal also with pseudo $G_{\delta}$
sets. Suppose $Y\subseteq\C^{\alpha}$ is a pseudo $G_{\delta}$ set.
For a countable model $M$, $Y_{M}=r_{\alpha,M}\left(Y\right)\subseteq S_{\alpha}\left(M\right)$
is not necessarily $G_{\delta}$. But in case $Y$ is closed under
$\equiv_{L}^{\alpha}$, it is. Indeed, $\C^{\alpha}\backslash Y$
is pseudo $F_{\sigma}$, and so $r_{\alpha,M}\left(\C^{\alpha}\backslash Y\right)$
is $F_{\sigma}$. But by Remark \ref{rem:distance 2}, $r_{\alpha,M}\left(\C^{\alpha}\backslash Y\right)\cap Y_{M}=\emptyset$. 

For a countable model $M$, $Y_{M}$ is a Polish space (as every $G_{\delta}$
set is). In addition, changing the model does not change the Borel
cardinality:
\begin{prop}
\label{prop:G_delta changing the model} Fix a pseudo $G_{\delta}$
set $Y\subseteq\C^{\alpha}$, closed under $\mathordi{\equiv_{L}^{\alpha}}$.
Then
\[
\mathordi{\equiv_{L}^{\alpha,M}}\upharpoonright Y_{M}\sim_{B}\mathordi{\equiv_{L}^{\alpha,N}}\upharpoonright Y_{N}.
\]
\end{prop}
\begin{proof}
The proof is exactly the same as in \cite[Propositions 2.3, 2.6]{PillayKrupinskiSolecki},
but we repeat it for completeness. It is enough to establish this
when $M\subseteq N$. Let $\pi:S_{\alpha}\left(N\right)\to S_{\alpha}\left(M\right)$
be the restriction map. Then $\pi$ is a continuous map that reduces
$\mathordi{\equiv_{L}^{\alpha,N}}$ to $\mathordi{\equiv_{L}^{\alpha,M}}$.
By \cite[Fact 1.7 (i)]{PillayKrupinskiSolecki} there is a Borel section,
i.e., a Borel function $\pi':S_{\alpha}\left(M\right)\to S_{\alpha}\left(N\right)$
such that $\pi\circ\pi'=\id$. Now it follows that $\pi$ and $\pi'$
restricted to $Y_{M}$ and $Y_{N}$ witness Borel bi-reducibility. 
\end{proof}
This allows us to refer to the Borel cardinality of $\equiv_{L}^{\alpha}\upharpoonright Y$
without specifying the model.

\subsubsection{Countable or uncountable language.}

Let $T$ be any complete first order theory and $\alpha$ any ordinal.
In order to state our theorem in full generality, we shall need the
following definition:
\begin{defn}
We say that a set $Y\subseteq\C^{\alpha}$ for some small $\alpha$
is \emph{pseudo strong Choquet} if $Y_{M}$ is strong Choquet for
all $M$.\end{defn}
\begin{example}
Pseudo closed and pseudo $G_{\delta}$ sets which are closed under
$\equiv_{L}^{\alpha}$ are pseudo strong Choquet by the observation
after Fact \ref{fac:Change the model} and Proposition \ref{prop:G_delta subset is Choquet}. \end{example}
\begin{rem}
For countable $T$ and $\alpha$, ``pseudo strong Choquet'' is the
correct analog of pseudo $G_{\delta}$ for sets closed under $\equiv_{L}^{\alpha}$.
By \cite[Theorem 8.17]{KechrisClassical} if $Y\subseteq\C^{\alpha}$
is such a set, then $Y$ is pseudo strong Choquet iff $Y$ is pseudo
$G_{\delta}$ iff $Y_{M}$ is Polish for every $M$. 
\end{rem}

\subsection{Results}

Our main theorem, proved in Section \ref{sec:The-general-case}, is:
\begin{main_thm}
\label{MainThmA} Suppose $T$ is a complete countable first-order
theory, $\alpha$ a countable ordinal, and suppose $Y$ is a pseudo
$G_{\delta}$ subset of $\C^{\alpha}$ which is closed under $\mathordi{\equiv_{L}^{\alpha}}$.
If for some $a\in Y$, $\left[a\right]_{\mathordi{\equiv_{L}^{\alpha}}}$
is not $d$-bounded, then $\mathordi{\equiv_{L}^{\alpha}}\upharpoonright Y$
is non-smooth. \end{main_thm}
\begin{rem}
This theorem remains true also for many-sorted countable theories,
with the obvious adjustments. 
\end{rem}
We immediately get Conjecture 1 of \cite{PillayKrupinskiSolecki}:
\begin{cor}
\label{cor:Conjecture 1}Suppose $T$ and $\alpha$ are as above.
Suppose $K\subseteq\C^{\alpha}$ is a KP strong type. If $\mathordi{\equiv_{L}^{\alpha}}\upharpoonright K$
is not $d$-bounded, then $\mathordi{\equiv_{L}^{\alpha}}\upharpoonright K$
is non-smooth. In particular, by Remark \ref{rem:Lascar is not closed},
if $\mathordi{\equiv_{L}^{\alpha}}\upharpoonright K$ is not trivial,
then it is non-smooth. \end{cor}
\begin{proof}
Observe that if $\mathordi{\equiv_{L}^{\alpha}}\upharpoonright K$
is not $d$-bounded, then there is a $\mathordi{\equiv_{L}^{\alpha}}$-class
inside $K$ which is not $d$-bounded (else all classes will have
the same bound, since they are conjugates). \end{proof}
\begin{cor}
Suppose $T$ and $\alpha$ are as above. Then $\mathordi{\equiv_{L}^{\alpha}}$
is closed iff it is smooth. \end{cor}
\begin{proof}
If $\mathordi{\equiv_{L}^{\alpha}}$ is not closed, then $\mathordi{\equiv_{L}^{\alpha}}\neq\mathordi{\equiv_{KP}^{\alpha}}$,
so there is a KP strong type $K$ such that $\mathordi{\equiv_{L}^{\alpha}}\upharpoonright K$
is not trivial, so $\mathordi{\equiv_{L}^{\alpha}}\upharpoonright K$
is not smooth, so also $\mathordi{\equiv_{L}^{\alpha}}$. The other
direction follows from Proposition \ref{prop:Closed-equivalence-relations}. \end{proof}
\begin{rem}
Since our main result concerns $\EZero$, it actually says something
about the ``definable cardinality'' of $\mathordi{\equiv_{L}^{\alpha}}$,
i.e., it is stronger than just saying something about the Borel cardinality
of $\mathordi{\equiv_{L}^{\alpha}}$, but also allows reductions to
be ``definable''. In the proof of Proposition \ref{prop:E_0 is not definable smooth},
we showed that there are no separating sets for $\EZero$ with the
Baire property. In any reasonable interpretation of the term, any
``definable'' reduction of $\EZero$ to $\Delta\left(Y\right)$
for some Polish space $Y$ will give rise to such separating sets.
So our main result implies that the ``definable cardinality'' of
$\mathordi{\equiv_{L}^{\alpha}}$ is greater than $\Delta\left(2^{\omega}\right)$.
We will not give an exact definition of ``definable cardinality''
(see more in \cite[Chapter 8]{BeckerKechris}).
\end{rem}
For a general language and $\alpha$ we have:
\begin{main_thm}
\label{MainThmB} {[}Simplified version{]} Suppose $T$ is a complete
first-order theory, $\alpha$ a small ordinal. Suppose $Y\subseteq\C^{\alpha}$
is closed under $\mathordi{\equiv_{L}^{\alpha}}$ and for some $a\in Y$,
$\left[a\right]_{\mathordi{\equiv_{L}^{\alpha}}}$ is not $d$-bounded.
Suppose $Y$ is pseudo strong Choquet. Then $\left|Y/\mathordi{\equiv_{L}^{\alpha}}\right|\geq2^{\aleph_{0}}$.
\end{main_thm}
The full theorem says a bit more, see \ref{Thm:MainThmB}. 
\begin{cor}
Fact \ref{fac:Newelski} holds for any theory $T$ and any small ordinal
$\alpha$.\end{cor}
\begin{proof}
(1), (2) and (3) follow immediately from Main Theorem \ref{MainThmB}.
(3) is also connected to Corollary \ref{cor:G_delta=00003D>smooth}. 

(4) Suppose $T$ is small. Let $n<\omega$, let $a$ be some tuple
of length $n$ and let $Y=S_{n}\left(a\right)$. This is a countable
Polish space. Thus every subset of $Y$ is $G_{\delta}$, in particular
the set 
\[
Q=\left\{ q\in S_{n}\left(a\right)\left|\,\forall b\models q\,\left(b\equiv_{L}^{n}a\right)\right.\right\} .
\]
 (which can also can also be defined with $\exists$). Let $M$ be
any countable model containing $a$. Then the restriction map $\pi:S\left(M\right)\to S\left(a\right)$
is continuous. Thus, $\pi^{-1}\left(Q\right)$ is also $G_{\delta}$.
But it is exactly the Lascar strong type of $a$ in $S\left(M\right)$.
By (3), this class is $d$-bounded, but then by Remark \ref{rem:Lascar is not closed}
$\mathordi{\equiv_{L}^{n}}\upharpoonright\left[a\right]_{\mathordi{\equiv_{KP}^{n}}}$
is trivial and hence $\mathordi{\equiv_{KP}^{n}}=\mathordi{\equiv_{L}^{n}}$. 
\end{proof}

\section{\label{sec:Descriptive-set-theoretic}Descriptive set theoretic lemmas }

\subsection{Polish spaces}

Given a group $\Gamma$ of homeomorphisms of a topological space $X$,
we use $E_{\Gamma}^{X}$ to denote the corresponding orbit equivalence
relation. Although the following fact can be seen as a consequence
of the proof of \cite[Theorem 3.4.5]{BeckerKechris}, for the sake
of completeness we provide a proof.
\begin{thm}
\label{theorem:meagerembedding} Suppose that $X$ is a perfect Polish
space, $\Gamma$ is a group of homeomorphisms of $X$ with a dense
orbit, and $R\subseteq X\x X$ is a meager set. Then there is a continuous,
injective homomorphism $\phi:2^{\omega}\to X$ from $\left(\EZero,\twiddle\EZero\right)$
to $\left(E_{\Gamma}^{X},\twiddle R\right)$. \end{thm}
\begin{proof}
We use $1_{\Gamma}$ to denote the identity element of $\Gamma$.
Given a natural number $n$ and a sequence $\left\langle \gamma_{i}\left|\, i<n\right.\right\rangle $
of elements of $\Gamma$, we use $\prod_{i<n}\gamma_{i}$ to denote
$1_{\Gamma}$ when $n=0$, and the product $\gamma_{0}\cdots\gamma_{n-1}$
when $n>0$. When $\left\langle \gamma_{i}\left|\, i<n\right.\right\rangle $
is constant with value $\gamma$, we also use $\gamma^{n}$ to denote
$\prod_{i<n}\gamma_{i}$.

As $X$ is perfect, the set of pairs of distinct points of $X$ is
comeager, so there is a decreasing sequence $\sequence{U_{n}}{n\in\Nn}$
of dense, irreflexive, open, symmetric subsets of $X\times X$ whose
intersection is disjoint from $R$. We will recursively construct
group elements $\gamma_{n}\in\Gamma$, with which we associate the
products $\gamma_{s}=\prod_{i<n}\gamma_{i}^{s(i)}$, for all $n\in\Nn$
and $s\in2^{n}$. We will simultaneously construct points $x_{n}\in X$
and open neighborhoods $X_{n}$ of $x_{n}$ with the following properties: 
\begin{enumerate}
\item $\overline{X_{n+1}}\subseteq X_{n}\cap(\gamma_{n}^{-1}\cdot X_{n})$. 
\item $\forall s\in2^{n+1}\,\diam\left(\gamma_{s}\cdot X_{n+1}\right)\le1/n$. 
\item $\forall s,t\in2^{n+1}\,\left(s\left(n\right)\neq t\left(n\right)\Rightarrow(\gamma_{s}\cdot X_{n+1})\x(\gamma_{t}\cdot X_{n+1})\subseteq U_{n}\right)$. 
\end{enumerate}
We begin by fixing an arbitrary point $x_{0}\in X$ and setting $X_{0}=X$.

Suppose now that $n\in\Nn$ and we have already found $\sequence{\gamma_{m}}{m<n}$,
$x_{n}$, and $X_{n}$. The fact that $\Gamma$ consists of homeomorphisms
then ensures that the set 
\[
V_{n}=\bigcap\set{\left(\gamma_{s}\x\gamma_{t}\right)^{-1}\left(U_{n}\right)}{\left(s,t\right)\in2^{n}\x2^{n}}
\]
 is dense and open, so the fact that $\Gamma$ has a dense orbit yields
$\gamma_{n}\in\Gamma$ and $x_{n+1}\in X_{n}\cap(\gamma_{n}^{-1}\cdot X_{n})$
for which $\left(x_{n+1},\gamma_{n}\cdot x_{n+1}\right)\in V_{n}$.
As $\Gamma$ consists of homeomorphisms and $U_{n}$ is symmetric,
there is an open neighborhood $X_{n+1}$ of $x_{n+1}$ satisfying
conditions (1) -- (3). This completes the recursive construction.

Conditions (1) and (2) ensure that we obtain a continuous function
$\phi:2^{\omega}\to X$ by setting $\phi(c)=\lim_{n\to\infty}\gamma_{c\restriction n}\cdot x_{n}$.
To see that $\phi$ is a homomorphism from $\EZero$ to $E_{\Gamma}^{X}$,
it is sufficient to observe that if $k\in\Nn$, $s\in2^{k}$, and
$y\in2^{\omega}$, then 
\[
\phi(s\concat y)=\lim_{n\to\infty}\gamma_{s\concat y\restriction n}\cdot x_{n}=\lim_{n\to\infty}\gamma_{s}\gamma_{\left(0\right)^{k}\concat y\restriction n}\cdot x_{n}=\gamma_{s}\cdot\phi\left(\left(0\right)^{k}\concat y\right).
\]
Observe now that if $y,z\in2^{\omega}$ and $y\left(n\right)\neq z\left(n\right)$,
then conditions (1) and (3) ensure that $\left(\phi\left(y\right),\phi\left(z\right)\right)\in\left(\gamma_{y\restriction\left(n+1\right)}\cdot X_{n+1}\right)\times\left(\gamma_{z\restriction\left(n+1\right)}\cdot X_{n+1}\right)\subseteq U_{n}$,
so the irreflexivity of $U_{n}$ yields the injectivity of $\phi$,
and the fact that $\sequence{U_{n}}{n\in\Nn}$ is a decreasing sequence
whose intersection is disjoint from $R$ ensures that $\phi$ is a
homomorphism from $\twiddle\EZero$ to $\twiddle R$.
\end{proof}
Given $R\subseteq X\times X$ and $x\in X$, define $R_{x}=\set{y\in X}{x\mathrela Ry}$.
\begin{thm}
\label{theorem:embedding} Suppose that $X$ is a Polish space, $\sequence{R_{n}}{n\in\Nn}$
is a sequence of $F_{\sigma}$ subsets of $X\x X$, $\Gamma$ is a
group of homeomorphisms of $X$, and $\mathcal{O}\subseteq X$ is
an orbit of $\Gamma$ with the property that for all $n\in\Nn$ and
open sets $U\subseteq X$ intersecting $\mathcal{O}$, there are distinct
$x,y\in\calO\cap U$ with $\calO\cap\left(R_{n}\right)_{x}\cap\left(R_{n}\right)_{y}=\emptyset$.
Then there is a continuous, injective homomorphism $\phi:2^{\omega}\to\overline{\calO}$
from $\left(\EZero,\twiddle\EZero\right)$ to $\left(E_{\Gamma}^{X},\twiddle\bigcup\set{R_{n}}{n\in\Nn}\right)$. \end{thm}
\begin{proof}
In light of Theorem \ref{theorem:meagerembedding}, it is sufficient
to show that $\overline{\calO}$ is perfect and $\bigcup\set{R_{n}}{n\in\Nn}\upharpoonright\overline{\calO}$
is meager. For the former, observe that if $U\subseteq X$ is an open
set intersecting $\overline{\calO}$ , then it intersects $\calO$,
so there are distinct $x,y\in\calO\cap U$. For the latter, it is
sufficient to check that each of the sets $\twiddle R_{n}\restriction\overline{\calO}$
is dense. Towards this end, suppose that $U,V\subseteq X$ are open
sets intersecting $\overline{\calO}$, and therefore $\calO$. Then
there exist $x,y\in\calO\cap U$ with $\calO\cap\left(R_{n}\right)_{x}\cap\left(R_{n}\right)_{y}=\emptyset$,
as well as $z\in\calO\cap V$, so $\neg x\mathrela{R_{n}}z$ or $\neg y\mathrela{R_{n}}z$,
thus $\twiddle R_{n}\cap\calO\cap\left(U\x V\right)\neq\emptyset$.
\end{proof}
We are going to apply this in our context via:
\begin{cor}
\label{cor:our situation} Let $T$ be a countable first order theory,
let $\alpha$ be a countable ordinal and $M$ a countable model. Let
$Y$ be a Polish subspace of $S_{\alpha}\left(M\right)$ that is closed
under $\equiv_{L}^{\alpha,M}$. Suppose that there is some $\xx\in Y$
such that for every open set $U\ni\xx$ and for all $N\in\Nn$, there
exist some $\sigma\in\Autf\left(\C\right)$ such that:
\begin{enumerate}
\item The automorphism $\sigma^{*}$ that $\sigma$ induces on $S_{\alpha}\left(M\right)$
fixes $Y$ setwise.
\item $\sigma^{*}\left(\xx\right)\in U$ and $N<d_{\alpha}\left(\sigma^{*}\left(\xx\right),\xx\right)$
(see Definition \ref{def:distance on types}). 
\end{enumerate}
Then there is a continuous, injective homomorphism $\phi:2^{\omega}\to Y$
from $\left(\EZero,\twiddle\EZero\right)$ to $\left(\mathordi{\equiv_{L}^{\alpha,M}},\twiddle\mathordi{\equiv_{L}^{\alpha,M}}\right)$.
In particular, $\mathordi{\equiv_{L}^{\alpha,M}}\upharpoonright Y$
is not smooth. \end{cor}
\begin{proof}
For $n<\omega$, let $R_{n}$ be the closed set $\set{\left(p,q\right)\in Y\x Y}{d_{\alpha}\left(p,q\right)\leq n}$.
Let $\Gamma$ be the group of homeomorphisms of $Y$ which are induced
by automorphisms in $\Autf\left(\C\right)$ which fix $Y$ setwise.
Let $\calO$ be the orbit of $\xx$ under $\Gamma$. 

Let $n\in\Nn$ and let $W$ be an open set which intersects $\calO$.
Then for some $\gamma\in\Gamma$, $\gamma\left(\xx\right)\in W$.
Let $U=\gamma^{-1}\left(W\right)$. Then for some $\sigma\in\Autf\left(\C\right)$,
$\sigma^{*}\in\Gamma$, $\sigma^{*}\left(\xx\right)\in U$ and $2n<d_{\alpha}\left(\sigma^{*}\left(\xx\right),\xx\right)$.
Let $x=\gamma\left(\xx\right)$ and $y=\gamma\sigma^{*}\left(\xx\right)$.
So $d_{\alpha}\left(x,y\right)=d_{\alpha}\left(\xx,\sigma^{*}\left(\xx\right)\right)>2n$
and so $x$ and $y$ are distinct and $\calO\cap\left(R_{n}\right)_{x}\cap\left(R_{n}\right)_{y}=\emptyset$. 
\end{proof}

\subsection{\label{sub:Choquet-spaces}Choquet spaces}

In order to prove Main Theorem \ref{MainThmB}, we have to work over
a model $M$ of possibly uncountable size, hence $S(M)$ is no longer
a Polish space. The idea is to mimic the proof of Main Theorem \ref{MainThmA},
i.e., construct step-by-step an embedding of $\mathbb{E}_{0}$. In
the countable case we use completeness at the limit stage, but here
we use the winning strategy in the (strong) Choquet game. 

The main observation is that Theorem \ref{theorem:meagerembedding}
has a natural analog in the Choquet context:
\begin{thm}
\label{theorem:meagerembedding-Choquet} Suppose that $X$ is regular
topological space, $\Gamma$ is a group of homeomorphisms of $X$
and $\calO$ an orbit of $\Gamma$ such that $X$ is Choquet over
$\calO$. Suppose that for $n<\omega$, $V_{n}\subseteq X\x X$ is
a $G_{\delta}$ subset such that $V_{n}\upharpoonright\overline{\calO}\x\overline{\calO}$
is dense. Then there is a map $\phi:2^{\omega}\to\SS\left(X\right)$
such that for every $y,z\in2^{\omega}$: 
\begin{itemize}
\item $\phi\left(y\right)$ is a non-empty closed $G_{\delta}$ subset of
$X$.
\item If $z\mathrela{\EZero}y$ then there is some $\gamma\in\Gamma$ such
that $\gamma\cdot\phi\left(z\right)=\phi\left(y\right)$.
\item If $\twiddle z\mathrela{\EZero}y$ then $\left(\mbox{\ensuremath{\phi}}\left(y\right)\times\mbox{\ensuremath{\phi}}\left(z\right)\right)\subseteq\bigcap_{n<\omega}V_{n}$.
\end{itemize}
\end{thm}
\begin{proof}
The proof follows along the lines of the proof of Theorem \ref{theorem:meagerembedding}.
The main difference is that in condition (2) in the construction,
instead of controlling the diameter of the open sets, one has to refine
them so that they obey the winning strategy of player B in the suitable
Choquet game over $\calO$. 
\end{proof}
From this we get the following analog of \ref{theorem:embedding}:
\begin{thm}
\label{theorem:embedding-Choquet} Suppose that $X$ is a topological
space, $\sequence{R_{n}}{n\in\Nn}$ is a sequence of $F_{\sigma}$
subsets of $X\x X$, $\Gamma$ is a group of homeomorphisms of $X$,
and $\mathcal{O}\subseteq X$ is an orbit of $\Gamma$ with the property
that for all $n\in\Nn$ and open sets $U\subseteq X$ intersecting
$\mathcal{O}$, there are distinct $x,y\in\calO\cap U$ with $\calO\cap\left(R_{n}\right)_{x}\cap\left(R_{n}\right)_{y}=\emptyset$.
If $X$ is strong Choquet over $\calO$ then the conclusion of Theorem
\ref{theorem:meagerembedding-Choquet} holds with $V_{n}=\twiddle R_{n}$. 
\end{thm}
And:
\begin{cor}
\label{cor:our situation-Choquet} Let $T$ be any first order theory
with language $L$, let $\alpha$ be any ordinal and $M$ a model.
Let $Y$ be a subspace of $S_{\alpha}\left(M\right)$ that is closed
under $\equiv_{L}^{\alpha,M}$. Suppose that there is
\begin{enumerate}
\item Some $\xx\in Y$. 
\item A countable sub-language $L'$ of $L$, a countable set $M'\prec M\upharpoonright L'$
and a countable sub-tuple of the first $\alpha$ variables which for
simplicity we will assume to be the initial segment of length $\beta$.
\item A countable subgroup $\Sigma\leq\Autf\left(\C\right)$ of automorphisms
that fix $M'$ and $M$ setwise.
\end{enumerate}
Such that:
\begin{enumerate}
\item With the topology induced on $Y$ by $L'$, $M'$, and $\beta$ (the
one generated by formulas in $L'_{\beta}\left(M'\right)$), $Y$ is
strong Choquet over $\Sigma\cdot\xx$.
\item For every open set $U\ni\xx$ in the induced topology and for all
$N\in\Nn$, there exist some $\sigma\in\Sigma$ such that $\sigma^{*}\left(\xx\right)\in U$
and, letting $\xx'=\xx\upharpoonright L_{\beta}\left(M'\right)$,
$N<d_{\beta}'\left(\sigma^{*}\left(\xx'\right),\xx'\right)$ ($d_{\beta}'$
is the Lascar metric of the language $L'$).
\end{enumerate}
Then there is a map $\phi:2^{\omega}\to\SS\left(Y\right)$ such that
for every $y,z\in2^{\omega}$: 
\begin{itemize}
\item $\phi\left(y\right)$ is a non-empty closed $G_{\delta}$ subset of
$Y$.
\item If $z\mathrela{\EZero}y$ then there is a some $\gamma\in\Gamma$
such that $\gamma\cdot\phi\left(z\right)=\phi\left(y\right)$.
\item If $\twiddle z\mathrela{\EZero}y$ then $\left(\mbox{\ensuremath{\phi}}\left(y\right)\times\mbox{\ensuremath{\phi}}\left(z\right)\right)\cap\mathordi{\equiv_{L}^{\alpha,M}}=\emptyset$.
\end{itemize}
\end{cor}
\begin{proof}
Similar to  \ref{cor:our situation}. Note that if $p,q\in S_{\alpha}\left(M\right)$
and $p\upharpoonright L'_{\beta}\left(M'\right),q\upharpoonright L'_{\beta}\left(M'\right)$
are not $\mathordi{\equiv_{L}^{\beta,M'}}$ equivalent, then $p,q$
are not $\mathordi{\equiv_{L}^{\alpha,M}}$-equivalent. 
\end{proof}
The following lemma will not be used directly, but its proof will
give insight into the proof of Main theorem \ref{MainThmB}. 
\begin{lem}
\label{lem:Choquet - countable basis}Suppose $\left(X,\tau\right)$
is a Choquet space with topology $\tau$. Let $B\subseteq\SS\left(X\right)$
be a base for $\tau$, and assume it is closed under finite intersections.
Let $B_{0}\subseteq B$. Then there exists $B_{0}\subseteq B_{1}\subseteq B$
such that $\left|B_{1}\right|\leq\left|B_{0}\right|+\aleph_{0}$ and
$\left(X,\tau_{B_{1}}\right)$ is Choquet, where $\tau_{B_{1}}$ is
the topology generated by $B_{1}$.\end{lem}
\begin{proof}
Let $St$ be a winning strategy for player B in the Choquet game of
$\left(X,\tau\right)$. Let $s=\left\langle U_{i}\left|\, i\leq n\right.\right\rangle $
be a finite sequence of elements of $B_{0}$. Suppose $s$ consists
of a legal $n+1$-play of player A, where player B plays his moves
according to $St$ for $i<n$. Let $V_{s}$ be a nonempty basic open
set contained in player B's play in the $n$'th round of his move
according to $St$. Let $B_{0}^{1}$ be the closure under finite intersections
of $B_{0}\cup\set{V_{s}}{s\in B_{0}^{<\omega}}$. This is a subset
of $B$. Repeat this recursively to construct $B_{0}^{n}$ for $n<\omega$,
and let $B_{1}=\bigcup_{n<\omega}B_{0}^{n}$. Then $B_{1}$ satisfies
the cardinality demand. Let us see that $\left(X,\tau_{B_{1}}\right)$
is Choquet. For this we must describe a winning strategy for player
B.

So suppose $\left\langle \left(U_{i},V_{i}\right)\left|\, i<n\right.\right\rangle $
is a legal play of the Choquet game in $\tau_{B_{1}}$ (where $U_{i}$
is played by player A and $V_{i}$ is played by player B), and player
A chooses $U_{n}$. Suppose that:
\begin{itemize}
\item There are basic open sets $U_{i}'\in B_{1}$ and open sets $V_{i}'\in\tau$
for $i<n$ such that $\left\langle \left(U'_{i},V_{i}'\right)\left|\, i<n\right.\right\rangle $
is a play of the Choquet game compatible with $St$, $U_{i}'\subseteq U_{i}$
and $V_{i}\subseteq V_{i}'$. 
\end{itemize}
Let $U_{n}'\in B_{1}$ be such that $U_{n}'\subseteq U_{n}$. There
is some $m<\omega$ such that $U_{i}'\subseteq B_{0}^{m}$ for all
$i<n$ and let $s=\left\langle U_{i}'\left|\, i\leq n\right.\right\rangle $.
By construction of $B_{0}^{m+1}$, $V_{s}\in B_{1}$ so let player
B play $V_{s}$. 

If this does not hold, let player B play any set. 

Now it is easy to see that if player B plays according to this strategy,
then he will win the game.
\end{proof}

\section{\label{sec:The-small-case}The small case}

Here we prove Main Theorem \ref{MainThmA} under the assumption that
a consequence of smallness holds, namely that the conclusion of Fact
\ref{fac:Newelski} (4) holds. This result is superseded by Theorem
\ref{thm:T is not smooth} in the next section, and the reader may
skip it if desired.

Assume that $T$ is a complete theory in a countable language $L$
and that $\C$ is a monster model for $T$.
\begin{claim}
\label{cla:extension so that sigma is an automorphism}Suppose that
$A$ is a countable set and that $\left\{ \sigma_{i}\left|\, i<\omega\right.\right\} $
is a set of automorphisms of $\C$. Then there is a countable model
$N\supseteq A$ such that $\sigma_{i}\upharpoonright N$ is an automorphism
of $N$ for all $i<\omega$.\end{claim}
\begin{proof}
Let $M_{0}$ be some model containing $A$, and for $n>0$, let $M_{n}$
be a countable model containing $\bigcup_{j\in\mathbb{Z},i<\omega}\sigma_{i}^{\left(j\right)}\left(M_{n-1}\right)$.
Let $N=\bigcup_{n<\omega}M_{n}$. \end{proof}
\begin{defn}
\label{def:nice models}Call a countable model $M$ of $T$ \emph{nice}
if the following conditions hold:
\begin{enumerate}
\item For every pair of finite tuples $a,b\in M^{k}$, if $a\equiv_{L}^{k}b$
then there is a Lascar strong automorphism $\sigma$ of $M$ (i.e.,
$\sigma\in\Autf\left(\C\right)\cap\Aut\left(M\right)$) that maps
$a$ to $b$. Moreover, $\sigma$ has minimal bound (see Fact \ref{fac:Lascar Strong types, automorphisms, connection}
(2)) among all automorphisms in $\Autf\left(\C\right)$ that map $a$
to $b$. 
\item For every finite tuple $a\in M^{k}$, and every $n<\omega$, if there
are $c_{1},c_{2}\in\C^{k}$ such that $c_{1}\equiv_{L}^{k}a\equiv_{L}^{k}c_{2}$
and $d_{k}\left(c_{1},c_{2}\right)>n$, then there are such $c_{1},c_{2}$
in $M^{k}$. 
\item For all finite tuples $a,b\in M^{k}$ and $a'\in M^{k'}$ and every
$n<\omega$, if there is some $b'$ such that $d_{k+k'}\left(a\concat a',b\concat b'\right)\leq n$,
then there is some such $b'\in M^{k'}$.
\end{enumerate}
\end{defn}
\begin{lem}
\label{lem:Nice-models-exist.}Nice models exist. Moreover, for every
countable set $A$, there is a nice model $M$ that contains it.\end{lem}
\begin{proof}
Let $M_{0}$ be any countable model containing $A$. Recursively choose
$M_{n+1}$ to satisfy (1)--(3) relative to $M_{n}$ (using Claim \ref{cla:extension so that sigma is an automorphism})
and set $M=\bigcup_{n<\omega}M_{n}$.
\end{proof}
Fix a countable ordinal $\alpha$ and a pseudo $G_{\delta}$ set $Y$.
Assume that:
\begin{assumption}
\label{ass:weak assumption }
\begin{enumerate}
\item $\alpha$ is infinite. 
\item The Lascar strong type of every finite sub-tuple of a tuple from $Y$
is $d$-bounded. 
\end{enumerate}
\end{assumption}
\begin{rem}
By Fact \ref{fac:Newelski} (4), if $T$ is small, then for finite
tuples, $\mathordi{\equiv_{KP}}=\mathordi{\equiv_{L}}$, so this assumption
is satisfied when $\alpha$ is infinite if $T$ is small, and Corollary
\ref{cor:Conjecture 1} is trivial for finite $\alpha$ (given Fact
\ref{fac:Newelski} (1)). \end{rem}
\begin{thm}
\label{thm:Main assumption A}Main Theorem \ref{MainThmA} holds under
Assumption \ref{ass:weak assumption }. 

Namely, suppose $\alpha$ and $Y$ are as above and for some $\bar{a}\in Y$,
$\left[\bar{a}\right]_{\mathordi{\equiv_{L}^{\alpha}}}$ is not not
$d$-bounded. Then $\mathordi{\equiv_{L}^{\alpha}}\upharpoonright Y$
is non-smooth.\end{thm}
\begin{proof}
Choose a nice model $M$ (by Lemma \ref{lem:Nice-models-exist.})
that contains $\bar{a}$. 

Let $\xx=p=\tp\left(\bar{a}/M\right)$. We shall show that corollary
\ref{cor:our situation} applies with $Y$ there being $Y_{M}$ (see
Proposition \ref{prop:G_delta changing the model}).

Suppose $U$ is some open set containing $p$, and $N$ is some number.
In general, $U$ has the form $\left[\varphi\right]$ for some $\varphi\in L_{\alpha}\left(M\right)$.
But in our case, since $M\models\varphi\left(\bar{a}\right)$, we
may replace $U$ with a smaller open neighborhood of $p$ defined
by a formula of the form $x=c$, where $x$ is the tuple of first
$k$ variables and $c=\bar{a}\upharpoonright k$. 

Let $B$ be a bound on the diameter of $\left[c\right]_{\mathordi{\equiv_{L}^{k}}}$.
Since the class of $\bar{a}$ is not of bounded diameter, by compactness
there must be some finite extension of $c$ to a longer sub-tuple
$c\concat c'=\bar{a}\upharpoonright\left(k+k'\right)$ such that the
$\mathordi{\equiv_{L}^{k+k'}}$-class of $c\concat c'$ has diameter
greater than $2N+4B$. 

There are two tuples $f_{1}\concat f_{1}'$ and $f_{2}\concat f_{2}'$
in $\C$ and $\left[c\concat c'\right]{}_{\mathordi{\equiv_{L}^{k+k'}}}$
such that $d_{k+k'}\left(f_{1}\concat f_{1}',f_{2}\concat f_{2}'\right)>2N+2B$.
Since $M$ is nice, we may assume that these tuples are in $M$. 

By choice of $B$, niceness of $M$ and Remark \ref{rem:distance at most doubled},
there are $c''$ and $c'''$ in $M$ such that
\[
d_{k+k'}\left(f_{1}\concat f_{1}',c\concat c''\right),d_{k+k'}\left(f_{2}\concat f_{2}',c\concat c'''\right)\leq2B.
\]
So $d_{k+k'}\left(c\concat c'',c\concat c'''\right)>2N$. It follows
that for one of $c''$, $c'''$, say $c''$, $d_{k+k'}\left(c\concat c',c\concat c''\right)>N$
(but $c\concat c'\equiv_{L}^{k+k'}c\concat c''$).

Let $\sigma$ be a Lascar strong automorphism of $M$ that maps $c\concat c'$
to $c\concat c''$. Since $\sigma$ fixes $c$, $q=\sigma^{*}\left(p\right)\in U$.
But $q$ is realized by a tuple that contains $c\concat c''$, and
hence $d_{\alpha}\left(q,p\right)>N$. 
\end{proof}

\section{\label{sec:The-general-case}The countable case}

Assume that $\alpha$ is a countable ordinal, $T$ is a complete theory
in a countable language $L$ and $\C$ is a monster model for $T$.
\begin{defn}
For a formula $\alpha\left(x,a\right)$ over a tuple $a$ and an automorphism
$\sigma$, $\sigma\left(\alpha\right)=\alpha\left(x,\sigma\left(a\right)\right)$. 
\end{defn}

\begin{defn}
\label{def:class generic}Suppose $C\subseteq\C^{\alpha}$ is an $\mathordi{\equiv_{L}^{\alpha}}$-class.
A formula $\varphi\in L_{\alpha}\left(\C\right)$ is said to be $C$-generic
if finitely many translates of it under $\Autf\left(\C\right)$ cover
$C$. The formula $\varphi$ is said to be $C$-weakly generic if
there is a non-$C$-generic formula $\psi\in L_{\alpha}\left(\C\right)$
such that $\varphi\vee\psi$ is $C$-generic. A partial $p\subseteq L_{\alpha}\left(\C\right)$
is said to be $C$-generic ($C$-weakly generic) if all its formulas
are.\end{defn}
\begin{claim}
\label{cla:ideal}The formulas which are not $C$-weakly generic form
an ideal.\end{claim}
\begin{proof}
Suppose $\varphi_{1},\varphi_{2}$ are not $C$-weakly generic and
we have to show that $\varphi_{1}\vee\varphi_{2}$ is also not $C$-weakly
generic. If not, there is some non-$C$-generic $\psi$ such that
$\varphi_{1}\vee\varphi_{2}\vee\psi$ is $C$-generic. But $\varphi_{2}\vee\psi$
is not $C$-generic (since $\varphi_{2}$ is not $C$-weakly generic),
so we get a contradiction.
\end{proof}
By $\varphi\vdash_{C}\psi$ we mean that for every $a\in C$, if $\C\models\varphi\left(a\right)$
then $\C\models\psi\left(a\right)$. 
\begin{rem}
\label{rem:if phi is g then also psi}If $\varphi\vdash_{C}\psi$
and $\sigma\in\Autf\left(\C\right)$ then $\sigma\left(\varphi\right)\vdash_{C}\sigma\left(\psi\right)$,
so if $\varphi$ is (weakly) generic, then so is $\psi$. \end{rem}
\begin{defn}
\label{def:proper generic}Suppose $p$ is a weakly generic (partial)
type over $\C$. Suppose furthermore that $p$ is closed under conjunctions.
Say that it is \emph{$C$-proper} if there is a non-$C$-generic formula
$\psi$ such that for all $\varphi\in p$, $\varphi\vee\psi$ is $C$-generic.
In general, $p$ is $C$-proper when its closure under finite conjunctions
is.
\end{defn}
Fix an $\mathordi{\equiv_{L}^{\alpha}}$-class $C$. When we write
``(weakly) generic'' and ``proper'', we mean ``$C$-(weakly)
generic'' and ``$C$-proper''. 
\begin{example}
\label{exa:generics are proper} If $p$ is generic, then it is proper. 
\end{example}
An easy and well known combinatorial lemma is the following:
\begin{lem}
\label{lem:coloring co-final}If $\left(P,<\right)$ is a directed
order, $k<\omega$ and $f:P\to k$ is some function, then there is
a cofinal $f$-homogeneous set $P_{0}\subseteq P$: there is some
$i<k$ such that $f^{-1}\left(i\right)$ is cofinal. \end{lem}
\begin{proof}
Suppose not. So for each $i<k$, the $f^{-1}\left(i\right)$ is not
cofinal, for some $p_{i}\in P$, for no $q\geq p_{i}$, $f\left(q\right)=i$.
Let $p$ be $\geq p_{i}$ for every $i<k$. Then $p\geq p_{f\left(p\right)}$
--- contradiction. \end{proof}
\begin{lem}
\label{lem:basic adding formula formula}Suppose $p\subseteq L_{\alpha}\left(\C\right)$
is a partial proper type as witnessed by $\psi$. Suppose that $\bigvee_{i<n}\varphi_{i}\vee\psi'$
covers $C$ and that $\psi'\vee\psi$ is non-generic. Then for some
$i<n$, $p\cup\left\{ \varphi_{i}\right\} $ is proper.\end{lem}
\begin{proof}
We may assume that $p$ is closed under conjunctions. For each formula
$\zeta\in p$, by assumption we have: 
\[
\zeta\vee\psi\vdash_{C}\bigvee_{i<n}\left(\varphi_{i}\wedge\zeta\right)\vee\psi'\vee\psi.
\]

So by Remark \ref{rem:if phi is g then also psi}, the right hand
side is generic. 

For each $\zeta\in p$ and $k<n$, let $\zeta_{k}=\bigvee_{k\leq i<n}\left(\varphi_{i}\wedge\zeta\right)\vee\psi'\vee\psi$.
Let $k_{\zeta}<n$ be maximal such that $\zeta_{k}$ is generic (must
exist since $\zeta_{0}$ is generic), so $\zeta_{k+1}$ is non-generic.
By Lemma \ref{lem:coloring co-final}, for some $k<n$, the set $\set{\zeta}{k_{\zeta}=k}$
is cofinal in the order $\zeta_{1}>\zeta_{2}\Leftrightarrow\zeta_{1}\vdash\zeta_{2}$.
Fix some $\chi\in p$ such that $k_{\chi}=k$. We will show that $p\cup\left\{ \varphi_{k}\right\} $
is proper, as witnessed by $\chi_{k+1}$. 

Suppose $\zeta\in p$. Let $\zeta'=\zeta\wedge\chi$, and $\zeta''\vdash\zeta'$
be such that $k_{\zeta''}=k$. Then $\left(\zeta''\wedge\varphi_{k}\right)\vee\zeta''_{k+1}$
is generic. Since $\zeta''\wedge\varphi_{k}\vdash\zeta\wedge\varphi_{k}$
and $\zeta''_{k+1}\vdash\chi_{k+1}$, $\left(\zeta\wedge\varphi_{k}\right)\vee\chi_{k+1}$
is also generic and we are done.
\end{proof}

\begin{lem}
\label{lem:proper extension for one formula}If $p\subseteq L_{\alpha}\left(\C\right)$
is a partial proper type, then for every formula $\varphi\in L_{\alpha}\left(\C\right)$,
either $p\cup\left\{ \varphi\right\} $ is proper or $p\cup\left\{ \neg\varphi\right\} $
is proper. \end{lem}
\begin{proof}
Apply Lemma \ref{lem:basic adding formula formula} with $n=2$, $\varphi_{0}=\varphi$,
$\varphi_{1}=\neg\varphi$ and $\psi'=\bot$ (i.e., $\forall x\left(x\neq x\right)$). 

Note that if we do not care about properness but only about weak genericity,
then this follows directly from Claim \ref{cla:ideal}. \end{proof}
\begin{prop}
\label{prop:proper extension with automorphisms}Suppose that $p\subseteq L_{\alpha}\left(\C\right)$
is a partial proper type, and $\varphi\in p$. Then there are $\sigma_{0},\ldots\sigma_{n-1}\in\Autf\left(\C\right)$
such that for every $\sigma\in\Autf\left(\C\right)$, there exists
some $i<n$ such that $p\cup\left\{ \sigma\left(\sigma_{i}\left(\varphi\right)\right)\right\} $
is proper. \end{prop}
\begin{proof}
Since $p$ is proper, there is some non-generic formula $\psi\left(x\right)$
that witnesses it. In particular, there is some $n<\omega$ and some
$\sigma_{0},\ldots,\sigma_{n-1}\in\Autf\left(\C\right)$ such that
$\bigvee_{i<n}\sigma_{i}\left(\varphi\vee\psi\right)$ covers $C$. 

Suppose that $\sigma\in\Autf\left(\C\right)$. Then $\sigma\left(\bigvee_{i<n}\sigma_{i}\left(\varphi\vee\psi\right)\right)=\bigvee_{i<n}\sigma\left(\sigma_{i}\left(\varphi\vee\psi\right)\right)$
also covers $C$. Since $\psi$ is non-generic, $\psi'=\bigvee_{j<n}\sigma\left(\sigma_{j}\left(\psi\right)\right)$
is also non-generic and so is $\psi'\vee\psi$. 

Now we can apply Lemma Lemma \ref{lem:basic adding formula formula}.\end{proof}
\begin{prop}
\label{prop:d(x,a)<=00003D1 is generic}Let $a\in C$. Consider the
partial type $q\left(x,y\right)=d_{\alpha}\left(x,y\right)\leq1$.
Then $q\left(x,a\right)\subseteq L_{\alpha}\left(a\right)$ is generic
and hence proper.\end{prop}
\begin{proof}
Suppose $\varphi\left(x,a\right)$ is a formula in $q\left(x,a\right)$.
Suppose $\varphi$ is non-generic. This means that for every $n$
Lascar strong conjugates $a_{0},\ldots,a_{n-1}$ of $a$, there is
some $a'\in C$ (so another Lascar conjugate of $a$) such that $\neg\varphi\left(a',a_{i}\right)$
holds for all $i<n$. Thus we can construct an infinite sequence $\left\langle a_{i}\left|\, i<\omega\right.\right\rangle $
of Lascar conjugates of $a$ such that for every $i<\omega$, $\neg\varphi\left(a_{i},a_{j}\right)$
holds for all $j<i$. 

By Fact \ref{fac:indiscernibles exist} (1), there is an indiscernible
sequence $\left\langle b_{i}\left|\, i<\omega\right.\right\rangle $
such that for all $j<i<\omega$, $\neg\varphi\left(b_{i},b_{j}\right)$
holds. But this is a contradiction because by definition $d_{\alpha}\left(b_{1},b_{0}\right)\leq1$.\end{proof}
\begin{thm}
\label{thm:T is not smooth}Main Theorem \ref{MainThmA} holds: 

Suppose $T$ is a complete countable first-order theory, $\alpha$
a countable ordinal, and suppose $Y$ is a pseudo $G_{\delta}$ subset
of $\C^{\alpha}$ which is closed under $\mathordi{\equiv_{L}^{\alpha}}$.
If for some $a\in Y$, $\left[a\right]_{\mathordi{\equiv_{L}^{\alpha}}}$
is not $d$-bounded, then $\mathordi{\equiv_{L}^{\alpha}}\upharpoonright Y$
is non-smooth. \end{thm}
\begin{proof}
Let $C=\left[a\right]_{\mathordi{\equiv_{L}^{\alpha}}}$. For what
follows when we write proper, we mean $C$-proper. 

We want to apply Corollary \ref{cor:our situation} with $Y$ there
being $Y_{M}$ (See Proposition \ref{prop:G_delta changing the model})
for some countable model $M$. Hence we will construct a pair $\left(M,p\right)$
such that $M$ is a countable model of $T$ and $\xx=p\in Y_{M}$
satisfies the condition in Corollary \ref{cor:our situation}. Translating,
this means that for every formula $\varphi\in p$, and every $N<\omega$,
there exists some Lascar strong automorphism $\sigma$ such that $\sigma\left(M\right)=M$,
$\varphi\in\sigma\left(p\right)$ and $d_{\alpha}\left(\sigma\left(p\right),p\right)>N$. 

Let $q\left(x\right)=d_{\alpha}\left(x,a\right)\leq1$ (as in Proposition
\ref{prop:d(x,a)<=00003D1 is generic}). We construct a sequence $\left\langle \sigma_{i},p_{i},M_{i}\left|\, i<\omega\right.\right\rangle $
such that:
\begin{enumerate}
\item $M_{i}$ is a finite set for all $i<\omega$.
\item $M_{i}\subseteq M_{i+1}$ and $p_{i}\subseteq p_{i+1}$ for all $i<\omega$.
\item For all $i<\omega$, $p_{i}$ is a finite type over $M_{i}$ such
that $p_{i}\cup q$ is proper. 
\item For all $i<\omega$, $\sigma_{i}$ is a Lascar strong automorphism. 
\item For every $i<\omega$ and formula of $\varphi\in L_{1}\left(M_{i}\right)$,
if $\varphi$ is not empty then for some $i<j<\omega$ there is some
$c\in M_{j}$ such that $\varphi\left(c\right)$ holds. 
\item For every $i<\omega$ and $n<\omega$ there exists some $i<j<\omega$
such that $M_{j}$ contains $\bigcup_{-n<l<n,i'<i}\sigma_{i'}^{\left(l\right)}\left(M_{i}\right)$. 
\item For every $i<\omega$ and formula $\varphi\in L_{\alpha}\left(M_{i}\right)$,
there is some $i<j<\omega$ such that $p_{j}$ contains either $\varphi$
or $\neg\varphi$. 
\item For every $i<\omega$, $N<\omega$ and $\varphi\in p_{i}$ there are
some $i<j<j'<\omega$ such that $d_{\alpha}\left(\sigma_{j}\left(a\right),a\right)>N$
and $\sigma_{j}^{-1}\left(\varphi\right)\in p_{j'}$. 
\end{enumerate}
If we succeed, then let $M=\bigcup M_{i}$, $p=\bigcup p_{i}$. $M$
is a model by (5), and $p\in S\left(M\right)$ and even belongs to
$Y_{M}$ by (7) and (3). 

(3), (6) and (8) imply that $\left(M,p\right)$ satisfy the required
condition: for every formula $\varphi\left(x\right)\in p$ (i.e.,
$p\in\left[\varphi\right]$) and $N<\omega$, there is some $\sigma_{j}$
as in (8). By (3), $d_{\alpha}\left(c,a\right)\leq3$ for any $c\models p$
(because there exists some $c'\models p\cup q$, and $d_{\alpha}\left(c,c'\right)\leq2$).
So $d_{\alpha}\left(c,\sigma_{j}\left(a\right)\right)>N-3$. By (6)
$\sigma_{j}\left(M\right)=M$, and so $d_{\alpha}\left(p,\sigma_{j}^{*}\left(p\right)\right)$
is well defined and $>N-6$ (as for any $c\models\sigma_{j}^{*}\left(p\right)$,
$d_{\alpha}\left(c,\sigma_{j}\left(a\right)\right)\leq3$). Finally,
$\varphi\in\sigma_{j}^{*}\left(p\right)$. 

The construction:

Let $M_{0}$ and $p_{0}$ be $\emptyset$. Note that condition (3)
holds by Proposition \ref{prop:d(x,a)<=00003D1 is generic}. 

Now we partition the work so that can satisfy all conditions. In each
stage we take care of one of (5)--(8). 

(5) and (6) are easy (just add some elements to $M_{i}$). (7) can
be achieved by Lemma \ref{lem:proper extension for one formula}. 

For (8) we need some argument. So suppose we are in stage $i+1$ of
the construction and we deal with (8), i.e., we are given $N<\omega$
and $\varphi\in p_{i}$. By Proposition \ref{prop:proper extension with automorphisms},
there are $\tau_{0},\ldots\tau_{n-1}\in\Autf\left(\C\right)$ such
that for every $\sigma\in\Autf\left(\C\right)$, there exists some
$j<n$ such that $q\cup p_{i}\cup\left\{ \sigma\left(\tau_{j}\left(\varphi\right)\right)\right\} $
is proper. There is some bound $k$ on $\tau_{j}$ for all $j<n$.
Let $\sigma\in\Autf\left(\C\right)$ be such that $d_{\alpha}\left(a,\sigma\left(a\right)\right)>N+k$.
By the triangle inequality, $d_{\alpha}\left(a,\sigma\left(\tau_{j}\left(a\right)\right)\right)>N$
for all $j<n$. For some $j<n$, $q\cup p_{i}\cup\left\{ \sigma\left(\tau_{j}\left(\varphi\right)\right)\right\} $
is proper, so let $\sigma_{i+1}=\left(\sigma\circ\tau_{j}\right)^{-1}$
(note that $d_{\alpha}\left(a,\left(\sigma\circ\tau_{j}\right)\left(a\right)\right)=d_{\alpha}\left(a,\left(\sigma\circ\tau_{j}\right)^{-1}\left(a\right)\right)$)
and $p_{i+1}=p_{i}\cup\left\{ \sigma\left(\tau_{j}\left(\varphi\right)\right)\right\} $
and continue. 
\end{proof}

\section{\label{sec:Uncountable-language}The general case}

Here we adapt our techniques to the case where the language is not
necessarily countable. 
\begin{thm}
\label{Thm:MainThmB} Main theorem \ref{MainThmB} holds:

Suppose $T$ is a complete first-order theory, $\alpha$ a small ordinal.
Suppose $Y\subseteq\C^{\alpha}$ is closed under $\mathordi{\equiv_{L}^{\alpha}}$
and for some $a\in Y$, $\left[a\right]_{\mathordi{\equiv_{L}^{\alpha}}}$
is not $d$-bounded. Suppose $Y$ is pseudo strong Choquet. Then there
is a model $M$ of size $\left|T\right|+\left|\alpha\right|$ and
a function $\phi:2^{\omega}\to\SS\left(Y_{M}\right)$ such that for
every $y,z\in2^{\omega}$: 
\begin{itemize}
\item $\phi\left(y\right)$ is a non-empty closed $G_{\delta}$ subset of
$Y_{M}$.
\item If $z\mathrela{\EZero}y$ then there is a some $\gamma\in\Gamma$
such that $\gamma\cdot\phi\left(z\right)=\phi\left(y\right)$.
\item If $\twiddle z\mathrela{\EZero}y$ then $\left(\phi\left(y\right)\times\phi\left(z\right)\right)\cap\mathordi{\equiv_{L}^{\alpha,M}}=\emptyset$.
\end{itemize}
In particular, $\left|Y/\mathordi{\equiv_{L}^{\alpha}}\right|\geq2^{\aleph_{0}}$. \end{thm}
\begin{proof}
The idea is to simultaneously construct a countable language $L'$,
a countable model $M'$ , a countable sub-tuple of the first $\alpha$
variables, an $L'$-type over $M'$ in these variables and a countable
group of Lascar strong automorphisms so that we can apply Corollary
\ref{cor:our situation-Choquet}. Eventually, $\xx$ will be any completion
of the $L'$-type over $M'$ to a complete $L$-type over $M$. 

So we will need a more elaborate argument than the one used in Theorem
\ref{thm:T is not smooth} that will also use the proof of Lemma \ref{lem:Choquet - countable basis}
(but not the lemma itself). That is, we try to construct the winning
strategy along with the model and language. 

Let $C=\left[a\right]_{\mathordi{\equiv_{L}^{\alpha}}}$. For what
follows when we write proper, we mean $C$-proper. Fix a countable
set $S$ of Lascar strong automorphisms that witness that $C$ is
not $d$-bounded, i.e., such that for all $N>0$, there is some $\sigma\in S$
such that $d_{\alpha}\left(a,\sigma\left(a\right)\right)>N$.

Let $M$ be a model of $T$ of size $\left|T\right|+\left|\alpha\right|$
that contains $a$ such that every $\sigma\in S$ fixes $M$ setwise
and for every generic formula over $M$, there are Lascar strong automorphisms
that witness it which fix $M$ setwise. Such a model can be constructed
as in Claim \ref{cla:extension so that sigma is an automorphism}.
Let $\Gamma$ be the group of Lascar strong automorphisms that fix
$M$ setwise. Let $St$ be a strategy for player B that witnesses
that $Y_{M}$ is strong Choquet. 

We construct:
\begin{itemize}
\item A countable sub-language $L'\subseteq L$.
\item A countable model $M'\prec M\upharpoonright L'$.
\item A countable sub-tuple $x'$ of the first $\alpha$ variables. For
notational simplicity we will assume that $x'$ is the first $\beta$
variables for a countable ordinal $\beta$.
\item A complete $L'$-type $p$ over $M_{0}$ in $x'$ which is consistent
with the type $q\left(x\right)=d_{\alpha}\left(x,a\right)\leq1$.
\item A countable subgroup $\Sigma\subseteq\Gamma$ of automorphism that
fix $M'$ and $M$ setwise.
\item A countable set $Q$ of complete types in $S_{\alpha}\left(M\right)$
contained in $Y_{M}$. 
\end{itemize}
Such that:
\begin{enumerate}
\item For every formula $\varphi\in p$ and natural number $N$, there is
an automorphism $\sigma\in\Sigma$ such that $\sigma^{-1}\left(\varphi\right)\in p$
and $d_{\beta}'\left(\sigma^{*}\left(p\right),p\right)>N$ where $d'$
is the Lascar metric when restricted to $L'$.
\item For every $\sigma_{0},\ldots,\sigma_{n}\in\Sigma$, $r_{0},\ldots,r_{n-1}\in Q$
and every sequence of $L'_{\beta}\left(M'\right)$ formulas $\sequence{\left(\varphi_{i},\psi_{i}\right)}{i<n}$
and a formula $\varphi_{n}$ such that:

\begin{enumerate}
\item $\varphi_{i+1}\vdash\psi_{i}\vdash\varphi_{i}$ for all $i<n$.
\item $\psi_{i}\in\sigma_{i}^{*}\left(p\right)$ (in other words, $\sigma_{i}^{*}\left(p\right)$
is in the open set $\left[\psi_{i}\right]$). 
\item $r_{i}$ is a complete extension of $\left\{ \varphi_{i}\right\} $
consistent with $\sigma_{i}^{*}\left(q\right)$.
\item $\varphi_{n}\in\sigma_{n}^{*}\left(p\right)$.
\item For each $i<n$, $\psi_{i}$ is such that $\left[\psi_{i}\right]\cap Y_{M}$
is a basic open subset of player B's move according to $St$ in the
strong Choquet game where player A plays $\varphi_{i}$ and $r_{i}$.
\end{enumerate}

There is a type $r_{n}\in Q$ containing $\varphi_{n}$ and a formula
$\psi_{n}$ in $\sigma_{n}^{*}\left(p\right)\cap r_{n}$ contained
in $\varphi_{n}$ which is a subset of player B's move according to
$St$ in the strong Choquet game described in (e) where in the $n$'th
move player A chooses $\varphi_{n}$ and $r_{n}$. 

\end{enumerate}
For the construction we repeat the proof of Theorem \ref{thm:T is not smooth}
inside $M$. As there, we let $q=d\left(x,a\right)\leq1$, and note
that it is proper. The differences are:
\begin{itemize}
\item [$\star$] We choose our automorphisms from $\Gamma$ (this is no
problem, since they all come from witnesses of genericity of certain
formulas over $M$ composed with an element from $S$ by the proof
of \ref{prop:proper extension with automorphisms}). 
\item [$\star$] We have to take care of $d'$ instead of $d$. So in (8)
there we increase the language $L'$ so that not only $d_{\alpha}\left(\sigma_{j}\left(a\right),a\right)>N$
is true in $L$, but it also true in $L'$. Similarly add some variables
to the tuple of variables we construct so that this remains true when
restricted to these variables. 
\item [$\star$] We add a step to the construction that makes sure that
the set of automorphisms is a group. 
\item [$\star$] For (2), we add a step to the construction. We have to
take care of every choice of $\sigma_{0},\ldots,\sigma_{n}\in\Sigma$,
$r_{0},\ldots,r_{n-1}$, a formula $\varphi_{n}$ and a sequence of
$L'_{\beta}\left(M'\right)$ formulas $\sequence{\left(\varphi_{i},\psi_{i}\right)}{i<n}$
from the language and model constructed thus far that satisfy (a)--(e)
above. We may assume that $\varphi_{n}$ is the conjunction of $\sigma_{n}$
applied to the current finite partial type we have. For every complete
extension $r\in S_{\alpha}\left(M\right)$ of $\left\{ \varphi_{n}\right\} $
consistent with $\sigma_{n}^{*}\left(q\right)$, there is some open
set $r\in V_{r}\subseteq\left[\varphi_{n}\right]\cap Y_{M}$ that
player B plays according to $St$ in the strong Choquet game described
in (e) where in the $n$'th move player A chooses $\varphi_{n}$ and
$r$. Let $\psi_{r}$ be a formula in $L_{\alpha}\left(M\right)$
that contains $r$ (i.e., $\psi_{r}\in r$), $\psi_{r}\vdash\varphi_{n}$
and $\left[\psi_{r}\right]\cap Y_{M}$ is contained in $V_{r}$. It
follows that $\left\{ \varphi_{n}\right\} \cup\sigma_{n}^{*}\left(q\right)\vdash\bigvee\psi_{r}$.
By compactness and by Lemma \ref{lem:proper extension for one formula}
for some $r$, $\sigma_{n}^{*}\left(q\right)\cup\left\{ \varphi_{n},\psi_{r}\right\} $
is proper. So we may add $\sigma_{n}^{-1}\left(\psi_{r}\right)$ to
our partial type. Also we add the symbols appearing in $\psi_{r}$
to the language and the variables appearing in it to the tuple of
variables. Finally, add $r$ to $Q$. 
\end{itemize}
When the construction is done, it is easy to see that letting $\xx$
be any completion of the complete $L'_{\beta}\left(M'\right)$ type
constructed $p$, it satisfies all the demands of Corollary \ref{cor:our situation-Choquet}.
For instance, we need to check that with the topology induced on $Y_{M}$
by $L'$, $M'$, and $\beta$, $Y$ is strong Choquet over $\Sigma\cdot\xx$.
The point is that if player A chooses some basic open set $\left[\varphi\right]$
containing $\sigma^{*}\left(\xx\right)$ for some $\sigma\in\Sigma$,
then by construction there is some formula $\psi$ in $\sigma^{*}\left(p\right)$
(so $\left[\psi\right]$ contains $\sigma^{*}\left(\xx\right)$) that
is contained in $\left[\varphi\right]$ and some type $r_{0}\in Q$
such that $\psi$ is contained in player B's response to $\left(\left[\varphi\right],r_{0}\right)$.
So player B will now choose $\left[\psi\right]\cap Y_{M}$. So we
simulate a game in $Y_{M}$ in which player A chooses types from $Q$,
and player B responds by choosing a basic open subset of what $St$
says. Since $St$ was a winning strategy, the intersection must be
nonempty. 
\end{proof}
\bibliographystyle{alpha}
\bibliography{common}

\end{document}